\documentclass[10pt, reqno,amsmath,amsthm,amssymb,amscd]{amsart}
\usepackage{mathrsfs,amssymb, amscd,amsmath,amsthm}
\usepackage[enableskew,vcentermath]{youngtab}
\usepackage{multicol}\multicolsep=0pt
\usepackage{tikz}
\newcommand{\END}{\mathcal{E}nd}
\newcommand{\clr}{rgb:black,1;blue,4;red,1}
\newcommand{\wdot}{ node[circle, draw, fill=white, thick, inner sep=0pt, minimum width=4pt]{}}
\newcommand{\bdot}{ node[circle, draw, fill=\clr, thick, inner sep=0pt, minimum width=4pt]{}}

\newcommand{\ob}[1]{\mathsf{#1}}
\newcommand{\up}{\uparrow}
\newcommand{\down}{\downarrow}
\newcommand{\OBC}{\mathcal{OBC}}
\newcommand{\AOBC}{\mathcal{AOBC}}

\newcommand{\OB}{\mathcal{OB}}

\newcommand{\wrd}{\langle\up,\down\rangle}
\newcommand{\fulldot}{
    \begin{tikzpicture}[color=\clr]
    \draw (0,0) \bdot;
    \end{tikzpicture}
}
\newcommand{\emptydot}{
    \begin{tikzpicture}[color=\clr]
    \draw (0,0) \wdot;
    \end{tikzpicture}
}
\newcommand{\undot}[1]{\operatorname{undot}({#1})}
\newcommand{\p}[1]{|{#1}|}

\newcommand{\lcap}{
\begin{tikzpicture}[baseline = 3pt, scale=0.5, color=\clr]
        \draw[-,thick] (1,0) to[out=up, in=right] (0.53,0.5) to[out=left, in=right] (0.47,0.5);
        \draw[->,thick] (0.49,0.5) to[out=left,in=up] (0,0);
\end{tikzpicture}
}
\newcommand{\lcup}{
\begin{tikzpicture}[baseline = 6pt, scale=0.5, color=\clr]
        \draw[-,thick] (1,1) to[out=down, in=right] (0.53,0.5) to[out=left, in=right] (0.47,0.5);
        \draw[->,thick] (0.49,0.5) to[out=left,in=down] (0,1);
\end{tikzpicture}
}

\newcommand{\rcup}{
\begin{tikzpicture}[baseline = 6pt, scale=0.5, color=\clr]
        \draw[<-,thick] (1,1) to[out=down, in=right] (0.53,0.5) to[out=left, in=right] (0.47,0.5);
        \draw[-,thick] (0.49,0.5) to[out=left,in=down] (0,1);
\end{tikzpicture}
}
\newcommand{\swap}{
\begin{tikzpicture}[baseline = 3pt, scale=0.5, color=\clr]
        \draw[->,thick] (0,0) to[out=up, in=down] (1,1);
        \draw[->,thick] (1,0) to[out=up, in=down] (0,1);
\end{tikzpicture}
}

\newcommand{\dswap}{
\begin{tikzpicture}[baseline = 3pt, scale=0.5, color=\clr]
        \draw[<-,thick] (0,0) to[out=up, in=down] (1,1);
        \draw[<-,thick] (1,0) to[out=up, in=down] (0,1);
\end{tikzpicture}
}

\newcommand{\xdot}{
\begin{tikzpicture}[baseline = 3pt, scale=0.5, color=\clr]
        \draw[->,thick] (0,0) to[out=up, in=down] (0,1);
        \draw (0,0.4)\bdot;
\end{tikzpicture}
}

\newcommand{\xdotk}{
\begin{tikzpicture}[baseline = 3pt, scale=0.5, color=\clr]
        \draw[->,thick] (0,0) to[out=up, in=down] (0,1);
        \draw (0,0.4)\bdot;
        \draw(0.4,0.4)node{\tiny $k$};
\end{tikzpicture}
}

\newcommand{\xdotr}{
\begin{tikzpicture}[baseline = 3pt, scale=0.5, color=\clr]
        \draw[<-,thick] (0,0) to[out=up, in=down] (0,1);
        \draw (0,0.4)\bdot;
\end{tikzpicture}
}

\newcommand{\cldot}{
\begin{tikzpicture}[baseline = 3pt, scale=0.5, color=\clr]
        \draw[->,thick] (0,0) to[out=up, in=down] (0,1);
        \draw (0,0.4)\wdot;
\end{tikzpicture}
}

\def\SUM#1#2{{\mbox{$\sum\limits_{#1}^{#2}$}}}
 \providecommand{\og}{``}
\providecommand{\fg}{''} \providecommand{\smfandname}{and}

\def\sc{\scriptstyle}

\usepackage{amssymb}
\usepackage{mathrsfs,amssymb}
\usepackage{multicol}\multicolsep=0pt
\usepackage[enableskew,vcentermath]{youngtab}

\usepackage[sort]{cite}
\usepackage{xcolor,graphicx}

\def\crulefill{\leavevmode\leaders\hrule height 1pt\hfill\kern 0pt}
\long\def\QUERY#1{%
\leavevmode\newline%
\noindent$\star\star\star$\thinspace\textsf{Comment/Query}\crulefill\newline%
   \space #1\newline\hbox to 120mm{\crulefill}$\star\star\star$\newline}
\newtheorem{Theorem}{Theorem}[section]%[chapter] theorem number will %continue
\newtheorem{Lemma}[Theorem]{Lemma}
\newtheorem{Cor}[Theorem]{Corollary}
\newtheorem{Prop}[Theorem]{Proposition}

 \theoremstyle{definition}

\newtheorem{Defn}[Theorem]{Definition}

\newtheorem{conj}[Theorem]{Conjecture}
\newtheorem{rem}[Theorem]{Remark}

\numberwithin{equation}{section}
\theoremstyle{definition}
\makeatletter
\def\enumerate{\begingroup\ifnum\@enumdepth>3\@toodeep\else
      \advance\@enumdepth\@ne
      \edef\@enumctr{enum\romannumeral\the\@enumdepth}%
      \topsep\z@\parskip\z@
      \list{\csname label\@enumctr\endcsname}
        {\@nmbrlisttrue\let\@listctr\@enumctr
         \parsep\z@\itemsep\z@\topsep\z@
         \setcounter{\@enumctr}{0}
         \def\makelabel##1{\hss\llap{\rm ##1}}
       }\fi}

\makeatother

\let\bar=\overline
\let\epsilon=\varepsilon
\def\({\big(}
\def\){\big)}

\def\C{\mathbb C}

\def\Z{\mathbb Z}

\def\0{\underline{0}}

\DeclareMathOperator{\End}{End}

\def\Hom{\text{Hom}}

\def\U{\mathbf U}

{\catcode`\|=\active
  \gdef\set#1{\mathinner{\lbrace\,{\mathcode`\|"8000%
                                   \let|\midvert #1}\,\rbrace}}
  \gdef\seT#1{\mathinner{\Big\lbrace\,{\mathcode`\|"8000%
                                   \let|\midverT #1}\,\Big\rbrace}}
}
\def\midvert{\egroup\mid\bgroup}
\def\midverT{\egroup\,\Big|\,\bgroup}

\def\Set[#1]#2|#3|{\Big\{\ #2\ \Big| \
           \vcenter{\hsize #1mm\centering #3}\Big\}}

\def\d{\delta}

\def\l{\lambda}

\def\sc{\scriptstyle}

\def\dis{\displaystyle}

\def\Z{\mathbb{Z}}

\def\C{\mathbb{C}}

\def\es{\varepsilon}

\def\Hom{{\rm Hom}}

\def\OT#1{{\raisebox{-8.5pt}{$\stackrel{{\mbox{\Large$\dis\otimes$}}}{\sc#1}$}}}

\def\mfg{{\mathfrak g}}
\def\fh{{\mathfrak h}}

\def\Set{{\rm Set}}

\def\OTIMES{{{\sc\!}\otimes{\sc\!}}}

\def\Hom{\text{Hom}}%
\def\U{\mathbf U}%
\def\textsf#1{{\textit{#1}}}%

\begin{document}
\title{A proof of Comes-Kujawa's conjecture }
\author{ MengMeng Gao, Hebing Rui, Linliang Song, Yucai Su}
\address{M.G.  Department  of Mathematics, East China Normal University,  Shanghai, 2000240, China}\email{g19920119@163.com}
\address{H.R.  School of Mathematical Science, Tongji University,  Shanghai, 200092, China}\email{hbrui@tongji.edu.cn}
\address{L.S.  School of Mathematical Science, Tongji University,  Shanghai, 200092, China}\email{llsong@tongji.edu.cn}
\address{Y.S.  School of Mathematical Science, Tongji University,  Shanghai, 200092, China}\email{ycsu@tongji.edu.cn}
\thanks{H. Rui is supported  partially by NSFC (grant No.  11571108).  L. Song is supported  partially by NSFC (grant No.  11501368).  Y. Su is supported partially by NSFC (grant No. 11431010)}

\date{\today}
\sloppy \maketitle

\begin{abstract}Let $\kappa$ be a commutative ring containing $2^{-1}$. In this paper, we prove the
Comes-Kujawa's  conjecture on a $\kappa$-basis  of   cyclotomic oriented Brauer-Clifford supercategory.
As a by-product, we prove that the cyclotomic walled Brauer-Clifford superalgebra  defined by  Comes and Kujawa and ours are isomorphic if $\kappa$ is  an algebraically closed field with characteristic not two.
\end{abstract}

\section{Introduction}
The affine walled Brauer-Clifford supercategory and its cyclotomic quotients are introduced by Comes and Kujawa \cite{CK}. These supercategories have closed connections with  representations of queer Lie superalgebra $\mathfrak q(n)$, and its associated finite W-superalgebras, etc.
The aim of this paper  is to prove the  Comes-Kujawa's  conjecture on a basis  of   cyclotomic oriented Brauer-Clifford supercategory\cite[Conjecture~7.1]{CK}.

 Before we recall Comes-Kujawa's conjecture, we need some notions in \cite{CK} etc.
  Throughout, we assume that $\kappa$ is  an arbitrary commutative ring   containing $2^{-1}$.

\subsection{The affine walled Brauer-Clifford supercategory}
In this paper, we work over the super world. By definition,   a supermodule is a module  on which there is a   $\mathbb Z_2$-grading. We are going to  freely use the notions of
 $\kappa$-linear (monoidal) supercategories and superfunctors etc. For more details,  we  refer a reader to \cite{BE, CK}  and references therein.

For any two objects $\ob a,\ob b$ in a \emph{strict monoidal supercategory},  $\ob a\ob b$ represents  $\ob a\otimes \ob b$. So  $\ob a^k$ represents $\ob a\otimes \ldots\otimes \ob a$, where there are $k$ copies of $\ob a$ in the tensor product.
 Following  \cite{CK},   a  morphism $g:\ob a\to \ob b$ is drawn as
    $$
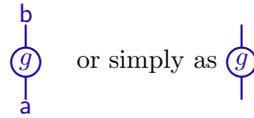
\begin{tikzpicture}[baseline = 12pt,scale=0.5,color=\clr,inner sep=0pt, minimum width=11pt]
        \draw[-,thick] (0,0) to (0,2);
        \draw (0,1) node[circle,draw,thick,fill=white]{$g$};
        \draw (0,-0.2) node{$\ob a$};
        \draw (0, 2.3) node{$\ob b$};
    \end{tikzpicture}
     \quad\text{ or simply as }
    \begin{tikzpicture}[baseline = 12pt,scale=0.5,color=\clr,inner sep=0pt, minimum width=11pt]
        \draw[-,thick] (0,0) to (0,2);
        \draw (0,1) node[circle,draw,thick,fill=white]{$g$};
    \end{tikzpicture}$$
 if there is no confusion for the objects. Note that $\ob a$ is drawn at the bottom while $\ob b$ is at the top.
 There is   a well-defined  tensor product of  two morphisms such that  $g\otimes h$  is given by horizontal stacking:
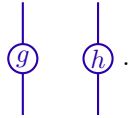
\begin{equation*}\begin{tikzpicture}[baseline = 19pt,scale=0.5,color=\clr,inner sep=0pt, minimum width=11pt]
        \draw[-,thick] (0,0) to (0,3);
        \draw[-,thick] (2,0) to (2,3);
        \draw (0,1.5) node[circle,draw,thick,fill=white]{$g$};
        \draw (2,1.5) node[circle,draw,thick,fill=white]{$h$};
        % \draw (0,-0.2) node{$\ob c$};
        % \draw (0, 3.3) node{$\ob d$};
        % \draw (2,-0.2) node{$\ob a$};
        % \draw (2, 3.3) node{$\ob b$};
    \end{tikzpicture}~.
 \end{equation*}
 To simplify the notation, the $r$-fold tensor products of  $g$ is drawn as
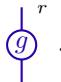
\begin{equation*}
    \begin{tikzpicture}[baseline = 12pt,scale=0.5,color=\clr,inner sep=0pt, minimum width=11pt]
        \draw[-,thick] (0,0) to (0,2);
        \draw (0,1) node[circle,draw,thick,fill=white]{$g$};
    \end{tikzpicture}^r
        ~.
\end{equation*}
 The composition of two morphisms  $g\circ h$ is given by vertical stacking:
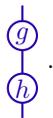
\begin{equation*}
        \begin{tikzpicture}[baseline = 19pt,scale=0.5,color=\clr,inner sep=0pt, minimum width=11pt]
        \draw[-,thick] (0,0) to (0,3);
        \draw (0,2.2) node[circle,draw,thick,fill=white]{$g$};
        \draw (0,0.8) node[circle,draw,thick,fill=white]{$h$};
        % \draw (0,-0.2) node{$\ob c$};
        % \draw (0.3,1.5) node{$\ob b$};
        % \draw (0, 3.3) node{$\ob a$};
    \end{tikzpicture}
    ~.
\end{equation*}
 Following \cite{CK}, a diagram  involving multiple products is  interpreted by \emph{first composing horizontally, then composing vertically}.
The  \emph{super-interchange law} is :
\[
    (g\otimes h)\circ(k\otimes l)=(-1)^{[h][k]}(g\circ k)\otimes(h\circ l).
\]
where $g, h, k, l$ are homogenous elements and  $[h]$ is the parity of $h$.

\begin{Defn}\cite[Definition~3.7]{CK}\label{aobc}
The \textsf{degenerate affine oriented Brauer-Clifford supercategory} $\AOBC_\kappa$
is  the $\kappa$-linear strict monoidal supercategory generated by two objects $\uparrow, \down$; four even morphisms $\lcup:1\rightarrow\uparrow\down, \ \ \lcap: \down\uparrow\rightarrow1, \ \ \swap: \uparrow\uparrow\rightarrow\uparrow\uparrow,\ \  \xdot:\uparrow\rightarrow\uparrow$; and one odd morphism $\cldot:\uparrow\rightarrow\uparrow$ subject to the following relations:
 \begin{equation}\label{OB relations 1 (symmetric group)}
        \begin{tikzpicture}[baseline = 10pt, scale=0.5, color=\clr]
            \draw[-,thick] (0,0) to[out=up, in=down] (1,1);
            \draw[->,thick] (1,1) to[out=up, in=down] (0,2);
            \draw[-,thick] (1,0) to[out=up, in=down] (0,1);
            \draw[->,thick] (0,1) to[out=up, in=down] (1,2);
        \end{tikzpicture}
        ~=~
        \begin{tikzpicture}[baseline = 10pt, scale=0.5, color=\clr]
            \draw[-,thick] (0,0) to (0,1);
            \draw[->,thick] (0,1) to (0,2);
            \draw[-,thick] (1,0) to (1,1);
            \draw[->,thick] (1,1) to (1,2);
        \end{tikzpicture}
        \ ,\qquad
        \begin{tikzpicture}[baseline = 10pt, scale=0.5, color=\clr]
            \draw[->,thick] (0,0) to[out=up, in=down] (2,2);
            \draw[->,thick] (2,0) to[out=up, in=down] (0,2);
            \draw[->,thick] (1,0) to[out=up, in=down] (0,1) to[out=up, in=down] (1,2);
        \end{tikzpicture}
        ~=~
        \begin{tikzpicture}[baseline = 10pt, scale=0.5, color=\clr]
            \draw[->,thick] (0,0) to[out=up, in=down] (2,2);
            \draw[->,thick] (2,0) to[out=up, in=down] (0,2);
            \draw[->,thick] (1,0) to[out=up, in=down] (2,1) to[out=up, in=down] (1,2);
        \end{tikzpicture}\ ,
    \end{equation}
    \begin{equation}\label{OB relations 2 (zigzags and invertibility)}
        \begin{tikzpicture}[baseline = 10pt, scale=0.5, color=\clr]
            \draw[-,thick] (2,0) to[out=up, in=down] (2,1) to[out=up, in=right] (1.5,1.5) to[out=left,in=up] (1,1);
            \draw[->,thick] (1,1) to[out=down,in=right] (0.5,0.5) to[out=left,in=down] (0,1) to[out=up,in=down] (0,2);
        \end{tikzpicture}
        ~=~
        \begin{tikzpicture}[baseline = 10pt, scale=0.5, color=\clr]
            \draw[-,thick] (0,0) to (0,1);
            \draw[->,thick] (0,1) to (0,2);\
        \end{tikzpicture}\
        ,\qquad
        \begin{tikzpicture}[baseline = 10pt, scale=0.5, color=\clr]
            \draw[-,thick] (2,2) to[out=down, in=up] (2,1) to[out=down, in=right] (1.5,0.5) to[out=left,in=down] (1,1);
            \draw[->,thick] (1,1) to[out=up,in=right] (0.5,1.5) to[out=left,in=up] (0,1) to[out=down,in=up] (0,0);
        \end{tikzpicture}
        ~=~
        \begin{tikzpicture}[baseline = 10pt, scale=0.5, color=\clr]
            \draw[-,thick] (0,2) to (0,1);
            \draw[->,thick] (0,1) to (0,0);
        \end{tikzpicture}\
        ,\qquad
        \begin{tikzpicture}[baseline = 10pt, scale=0.5, color=\clr]
            \draw[-,thick] (2,2) to[out=down, in=up] (2,1) to[out=down, in=right] (1.5,0.5) to[out=left,in=down] (1,1);
            \draw[->,thick] (1,1) to[out=up,in=right] (0.5,1.5) to[out=left,in=up] (0,1) to[out=down,in=up] (0,0);
            \draw[->,thick] (0.7,0) to[out=up,in=down] (1.3,2);
        \end{tikzpicture}
        ~\text{ is invertible},
    \end{equation}
 \begin{equation}\label{OBC relations}
        \begin{tikzpicture}[baseline = 7.5pt, scale=0.5, color=\clr]
            \draw[->,thick] (0,0) to[out=up, in=down] (0,1.5);
            \draw (0,1) \wdot;
            \draw (0,0.5) \wdot;
        \end{tikzpicture}
        ~=~
        \begin{tikzpicture}[baseline = 7.5pt, scale=0.5, color=\clr]
            \draw[->,thick] (0,0) to[out=up, in=down] (0,1.5);
        \end{tikzpicture}\
        ,\qquad
        \begin{tikzpicture}[baseline = 7.5pt, scale=0.5, color=\clr]
            \draw[->,thick] (0,0) to[out=up, in=down] (1,1.5);
            \draw[->,thick] (1,0) to[out=up, in=down] (0,1.5);
            \draw (0.2,0.5) \wdot;
        \end{tikzpicture}
        ~=~
        \begin{tikzpicture}[baseline = 7.5pt, scale=0.5, color=\clr]
            \draw[->,thick] (0,0) to[out=up, in=down] (1,1.5);
            \draw[->,thick] (1,0) to[out=up, in=down] (0,1.5);
            \draw (0.8,1) \wdot;
        \end{tikzpicture}\
        ,\qquad
        \reflectbox{\begin{tikzpicture}[baseline = -2pt, scale=0.5, color=\clr]
            \draw[->,thick] (0,0) to[out=up, in=left] (0.5,0.5);
            \draw[-,thick] (0.4,0.5) to (0.5,0.5) to[out=right,in=up] (1,0) to[out=down, in=right] (0.5,-0.5) to[out=left,in=down] (0,0);
            \draw (0,0) \wdot;
        \end{tikzpicture}}
        ~=0,
    \end{equation}

    \begin{equation}\label{AOBC relations}
        \begin{tikzpicture}[baseline = 7.5pt, scale=0.5, color=\clr]
            \draw[->,thick] (0,0) to[out=up, in=down] (0,1.5);
            \draw (0,1) \bdot;
            \draw (0,0.5) \wdot;
        \end{tikzpicture}
        ~=-~
        \begin{tikzpicture}[baseline = 7.5pt, scale=0.5, color=\clr]
            \draw[->,thick] (0,0) to[out=up, in=down] (0,1.5);
            \draw (0,1) \wdot;
            \draw (0,0.5) \bdot;
        \end{tikzpicture}\
        ,\qquad
        \begin{tikzpicture}[baseline = 7.5pt, scale=0.5, color=\clr]
            \draw[->,thick] (0,0) to[out=up, in=down] (1,1.5);
            \draw[->,thick] (1,0) to[out=up, in=down] (0,1.5);
            \draw (0.2,1) \bdot;
        \end{tikzpicture}
        ~-~
        \begin{tikzpicture}[baseline = 7.5pt, scale=0.5, color=\clr]
            \draw[->,thick] (0,0) to[out=up, in=down] (1,1.5);
            \draw[->,thick] (1,0) to[out=up, in=down] (0,1.5);
            \draw (0.8,0.5) \bdot;
        \end{tikzpicture}
        ~=~
        \begin{tikzpicture}[baseline = 7.5pt, scale=0.5, color=\clr]
            \draw[->,thick] (0,0) to[out=up, in=down] (0,1.5);
            \draw[->,thick] (1,0) to[out=up, in=down] (1,1.5);
        \end{tikzpicture}
        ~-~
        \begin{tikzpicture}[baseline = 7.5pt, scale=0.5, color=\clr]
            \draw[->,thick] (0,0) to[out=up, in=down] (0,1.5);
            \draw[->,thick] (1,0) to[out=up, in=down] (1,1.5);
            \draw (0,0.7) \wdot;
            \draw (1,0.7) \wdot;
        \end{tikzpicture}.
    \end{equation}
\end{Defn}
 Let $\wrd$ be the set of all words in the alphabets $\uparrow, \downarrow$ including
the empty word $\varnothing$. Each word $\ob a_1\ldots \ob a_r$ (resp., empty word $\varnothing$) represents  $\ob a_1\otimes \ldots \otimes \ob a_r$ (resp., the unit object $1$) in $\AOBC_\kappa$. The objects of previous  five morphisms are implicated in the pictures. In fact,  they can be read from the consistent orientation   of each strand.
For example, the objects  at both  the top and the bottom  of $\swap$ are   $\uparrow\uparrow$ since  the orientations of strands at both the top and the bottom  are  up-toward.
It means that $\swap$ is a morphism in $\End_{\AOBC_\kappa}(\uparrow\uparrow)$.
For $\lcup$, there is no endpoint at the bottom. It  means the object at the bottom is the unit object and hence $\lcup\in \Hom_{\AOBC_\kappa}(1,\uparrow\downarrow)$.
Similarly,  $\lcap\in \Hom_{\AOBC_\kappa}(\downarrow\uparrow,1)$ and $\xdot, \cldot\in\End_{\AOBC_\kappa}(\uparrow)$.
Since any morphism $g:\ob a\rightarrow\ob b$ in  $\AOBC_\kappa$ can be expressed as     tensor products and compositions of the five morphisms in Definition~\ref{aobc}, $\ob a,\ob b$ will be omitted when $g$ is drawn as a picture.
Given  an $\ob a\in \wrd$, following \cite{BCNR}, the identity morphism $\mathbf 1_{\ob a}\in \End_{\AOBC_k}(\ob a)$  can be drawn by the object $\ob a$ itself.
For example, $\mathbf 1_{\uparrow\uparrow\downarrow}=\uparrow\uparrow\downarrow$.

By~\cite[Definition~3.2]{CK}, the  oriented Brauer-Clifford supercategory $\OBC_\kappa$  is the subcategory of $\AOBC_\kappa$ generated by  the same objects, and  the previous morphisms except $ \xdot:\uparrow\rightarrow\uparrow$.  The oriented Brauer category $\OB_\kappa$\cite[Theorem~1.1]{BCNR} is the supercategory  generated by  the same objects and  the previous  morphisms except $ \cldot:\uparrow\rightarrow\uparrow$ and $ \xdot:\uparrow\rightarrow\uparrow$ subject to the relations \eqref{OB relations 1 (symmetric group)}--\eqref{OB relations 2 (zigzags and invertibility)}.

\subsection{Hom-superspaces of $\AOBC_\kappa$ and dotted oriented Brauer-Clifford diagrams with bubbles}
In order to state  bases  of  Hom-superspaces in  $\AOBC_\kappa$, $\OBC_\kappa$ and $\OB_\kappa$, we need to recall the definitions of (dotted) oriented Brauer(-Clifford) diagrams with bubbles in \cite{BCNR,CK}.

\begin{Defn}\cite{BCNR}
For any two words   $\ob a,\ob b\in\wrd$,
 an {\em oriented Brauer diagram} of type $\ob a \rightarrow \ob b$ is   an
oriented diagrammatic representation of a
 bijection
\begin{equation}\label{bijec}
\{i\:|\:\ob a_i = \up\} \sqcup \{i'\:|\:\ob b_i = \down\}
\stackrel{\sim}{\rightarrow}
\{i\:|\:\ob b_i = \up\}\sqcup\{i'\:|\:\ob a_i = \down\}
\end{equation}
obtained by placing $\ob a$ below
 $\ob b$, then drawing strands connecting pairs of letters as prescribed by the  bijection in \eqref{bijec}.
The consistent orientation to each strand in the diagram is given by the
letters of $\ob a$ and $\ob b$. Two oriented Brauer diagrams are {\em equivalent}
if they are of the same type and represent the same bijection.

\end{Defn}
 For example, the following is an oriented Brauer diagram of type $\downarrow^2\uparrow^3\rightarrow \downarrow^2\uparrow^3$.
 \begin{equation}\begin{tikzpicture}[baseline = 25pt, scale=0.35, color=\clr]
        \draw[<-,thick] (2,0) to[out=up,in=down] (0,5);
        \draw[->,thick] (6,0) to[out=up,in=down] (7,5);
        \draw[->,thick] (7,0) to[out=up,in=down] (6,5);
               \draw[<-,thick] (0,0) to[out=up,in=left] (2,1.5) to[out=right,in=up] (4,0);
        \draw[->,thick] (2,5) to[out=down,in=left] (3,4) to[out=right,in=down] (4,5);
       % \draw[->,thick] (2,3.1) to (2,3) to[out=down,in=right] (0.5,2) to[out=left,in=down] (-1,3)
%                        to[out=up,in=left] (0.5,4) to[out=right,in=up] (2,3);
   \end{tikzpicture}.
   \end{equation}
Following \cite{BCNR},  $\Delta_0:= \begin{tikzpicture}[baseline = -2pt, scale=0.5, color=\clr]
            \draw[->,thick] (0,0) to[out=up, in=left] (0.5,0.5);
            \draw[-,thick] (0.4,0.5) to (0.5,0.5) to[out=right,in=up] (1,0) to[out=down, in=right] (0.5,-0.5) to[out=left,in=down] (0,0);
        \end{tikzpicture}=\begin{tikzpicture}[baseline = -2pt, scale=0.5, color=\clr]
            \draw[-<,thick] (0,0) to[out=up, in=left] (0.5,0.5);
            \draw[-,thick] (0.4,0.5) to (0.5,0.5) to[out=right,in=up] (1,0) to[out=down, in=right] (0.5,-0.5) to[out=left,in=down] (0,0);
        \end{tikzpicture}$
is called  a \textsf{bubble}.
Two oriented Brauer diagrams with  bubbles are {\em equivalent}
if they have the same number of bubbles
and the underlying oriented Brauer diagrams without  bubbles are equivalent.
 It is proven  in \cite{BCNR} that two equivalent oriented Brauer diagrams with bubbles of type $\ob a \rightarrow \ob b$  represent the same morphism in $\Hom_{\OBC_\kappa} (\ob a, \ob b)$, and the set of   all equivalence classes of oriented Brauer diagrams  with bubbles of type $\ob a \rightarrow \ob b$ is a $\kappa$-basis of  $\Hom_{\OB_\kappa}(\ob a, \ob b)$.

\begin{Defn}\cite[\S3.3]{CK} For any two words  $\ob a,\ob b\in\wrd$,
an \emph{oriented Brauer-Clifford diagram}  (resp.,\ \emph{with bubbles}) of type $\ob a\rightarrow \ob b$  is an oriented Brauer diagram (resp.,~with bubbles) of type $\ob a\rightarrow \ob b$  such that there are  finitely many $\emptydot$'s on its segments.
A \emph{dotted oriented Brauer-Clifford diagram} (resp.,\ \emph{with bubbles}) is  an oriented Brauer diagram (resp.,\ with bubbles) such that there are  finitely many $\emptydot$'s and $\fulldot$'s on its segments.
\end{Defn}

 It is defined in  \cite{CK} (see also \cite{BCNR} for the second one) that

  \begin{equation}\label{down white dot}
    \begin{tikzpicture}[baseline = 12pt, scale=0.5, color=\clr]
        \draw[->,thick] (0,2) to (0,0);
        \draw (0,1) \wdot;
    \end{tikzpicture}
    ~:=~
    \begin{tikzpicture}[baseline = 12pt, scale=0.5, color=\clr]
        \draw[->,thick]
            (1,2) to (1,1)
            to[out=down,in=right] (0.75,0.25)
            to[out=left,in=down] (0.5,1)
            to[out=up,in=right] (0.25,1.75)
            to[out=left,in=up] (0,1) to (0,0);
        \draw (0.5,1) \wdot;
    \end{tikzpicture},\quad
    \text{  and }\quad
     \begin{tikzpicture}[baseline = 12pt, scale=0.5, color=\clr]
        \draw[->,thick] (0,2) to (0,0);
        \draw (0,1) \bdot;
    \end{tikzpicture}
    ~:=~
    \begin{tikzpicture}[baseline = 12pt, scale=0.5, color=\clr]
        \draw[->,thick]
            (1,2) to (1,1)
            to[out=down,in=right] (0.75,0.25)
            to[out=left,in=down] (0.5,1)
            to[out=up,in=right] (0.25,1.75)
            to[out=left,in=up] (0,1) to (0,0);
        \draw (0.5,1) \bdot;
    \end{tikzpicture}
\end{equation}
such that any morphism in $\Hom_{\OBC_\kappa}(\ob a,\ob b) $ (resp.,  $\Hom_{\AOBC_\kappa}(\ob a,\ob b) $)  can be realized as a $\kappa$-linear combination of (resp., dotted)  oriented Brauer-Clifford diagrams  with bubbles  of type $\ob a\rightarrow \ob b$.     In order to
give  bases  of  Hom-superspaces in $\OBC_\kappa$ (resp., $\AOBC_\kappa$), Comes and Kujawa introduced the  notion of a normally ordered (resp., dotted) oriented Brauer-Clifford diagram.
 \begin{Defn}\cite[\S~3.3]{CK}
 An oriented Brauer-Clifford diagram is called normally ordered if
\begin{enumerate}
    \item it has at most one $\emptydot$ on each strand and it has no bubble;
    \item all $\emptydot$'s are on outward-pointing boundary segments;
    \item all $\emptydot$'s are positioned at the same height if the  segments they occur on have the same orientation.
\end{enumerate}
\end{Defn}
 \begin{Defn}\cite[Definition~3.8]{CK}\label{normdot}
A dotted oriented Brauer-Clifford diagram with bubbles is \emph{normally ordered} if
\begin{enumerate}
    \item it is a normally ordered oriented Brauer-Clifford diagram by ignoring  all bubbles and all $\fulldot$'s;
   \item each $\fulldot$ is either on a bubble or on an inward-pointing boundary segment;
   \item each bubble has zero $\emptydot$'s, and an \emph{odd number} of $\fulldot$'s, are crossing-free, counterclockwise, and there are no other strands shielding it from the rightmost edge of the picture;
    \item whenever a $\fulldot$ and a $\emptydot$ appear on a segment that is both inward and outward-pointing, the $\emptydot$ appears ahead of the $\fulldot$ in the direction of the orientation.
\end{enumerate}
\end{Defn}
\noindent For example, the following diagrams represent  two  morphisms in $\Hom_{\AOBC_\kappa}(\down^2\up^3,\downarrow\uparrow^2)$. The right one is normally ordered whereas the left one is not.

\begin{equation}\label{two AOBC diagrams}
    \begin{tikzpicture}[baseline = 25pt, scale=0.35, color=\clr]
        \draw[<-,thick] (2,0) to[out=up,in=down] (0,5);
        \draw[->,thick] (6,0) to[out=up,in=down] (6,5);
        \draw[->,thick] (7,0) to[out=up,in=down] (7,5);
        \draw (7,1.53) \wdot;
        \draw[<-,thick] (0,0) to[out=up,in=left] (2,1.5) to[out=right,in=up] (4,0);
        %\draw[->,thick] (2,5) to[out=down,in=left] (3,4) to[out=right,in=down] (4,5);
        \draw[->,thick] (1,3.1) to (1,3) to[out=down,in=right] (-0.5,2) to[out=left,in=down] (-2,3)
                        to[out=up,in=left] (-0.5,4) to[out=right,in=up] (1,3);
        \draw (1.9,0.8) \bdot;
        \draw (3.7,0.8) \wdot;
        \draw (6,0.8) \bdot;
        \draw (6,1.53) \wdot;
        \draw (6,2.26) \bdot;
        \draw (-2,3) \bdot;
        %\draw (0.65,3) \wdot;
        \draw (6,3) \wdot;
        \draw (0.05,4.5) \bdot;
        %\draw (2.22,4.35) \wdot;
        %\draw (3.78,4.35) \bdot;
        \draw (6,4.35) \bdot;
    \end{tikzpicture}
    \qquad\qquad\qquad
    \begin{tikzpicture}[baseline = 25pt, scale=0.35, color=\clr]
        \draw[<-,thick] (2,0) to[out=up,in=down] (0,5);
        \draw[->,thick] (6,0) to[out=up,in=down] (6,5);
         \draw[->,thick] (6.7,0) to[out=up,in=down] (6.7,5);
         \draw (6.7,4.35) \wdot;
        \draw[<-,thick] (0,0) to[out=up,in=left] (2,1.5) to[out=right,in=up] (4,0);
       % \draw[->,thick] (2,5) to[out=down,in=left] (3,4) to[out=right,in=down] (4,5);
        \draw[->,thick] (8.1,4) to (8,4) to[out=left,in=up] (7,3) to[out=down,in=left] (8,2)
                        to[out=right,in=down] (9,3) to[out=up,in=right] (8,4);
        \draw (0.05,4.5) \bdot;
        \draw (0.65,3) \bdot;
        %\draw (2.22,4.35) \bdot;
%        \draw (3.78,4.35) \wdot;
        \draw (9,3) \bdot;
        \draw (6,5/4) \bdot;
        \draw (6,10/4) \bdot;
        \draw (6,15/4) \bdot;
        \draw (0.3,0.8) \wdot;
        \draw (1.9,0.8) \wdot;
    \end{tikzpicture}
\end{equation}

Let $\Delta_k=\begin{tikzpicture}[baseline = 5pt, scale=0.5, color=\clr]
        \draw[->,thick] (0.6,1) to (0.5,1) to[out=left,in=up] (0,0.5)
                        to[out=down,in=left] (0.5,0)
                        to[out=right,in=down] (1,0.5)
                        to[out=up,in=right] (0.5,1);
        \draw (1,0.5) \bdot;
        \draw (1.4,0.5) node{\footnotesize{$k$}};
    \end{tikzpicture}.~$  Then  $\Delta_k$ is the crossing-free and  counterclockwise bubble with $k$ $\fulldot$'s on it.  By \cite[Proposition~3.12]{CK},
 $\Delta_k=0$ whenever $k$ is even. This is the reason why Comes-Kujawa require that there are odd numbers of $\fulldot$'s on each bubble in Definition~\ref{normdot}. Later on, only  bubbles on which  there are odd $k$  $\fulldot$'s will be considered.

Following \cite{CK},  two normally ordered (resp., dotted) oriented Brauer-Clifford diagrams (resp., with bubbles)
are said to be \emph{equivalent} if their underlying oriented Brauer diagrams (resp., with bubbles) are equivalent and their corresponding strands have the same number of $\emptydot$'s (resp., and $\fulldot$'s).  It is proven  in \cite{CK} that two equivalent   normally ordered oriented Brauer-Clifford diagrams of type $\ob a \rightarrow \ob b$ represent the same morphism in $\Hom_{\OBC_\kappa} (\ob a, \ob b)$ and the set of  all equivalence  classes of  normally  ordered   oriented Brauer-Clifford diagrams of type $\ob a\rightarrow \ob b$
is a $\kappa$-basis of $\Hom_{\OBC_\kappa}(\ob a, \ob b)$. Unlike the cases for $\OB_\kappa$ and $\OBC_\kappa$, two equivalent normally ordered dotted oriented Brauer-Clifford diagrams (with bubbles) may not represent the same morphism in $\AOBC_\kappa$. The following is the main result of \cite{CK}.
\begin{Theorem}\cite[Corollary~6.4]{CK}\label{main-aff} For any  $\ob a, \ob b\in \wrd$, the set of all  equivalence classes of normally  ordered  dotted oriented Brauer-Clifford diagrams with bubbles of type $\ob a\rightarrow \ob b$
is   a $\kappa$-basis of $\Hom_{\AOBC_\kappa}(\ob a, \ob b)$. \end{Theorem}

\subsection{Cyclotomic quotients and Comes-Kujawa's conjecture}
Fix two nonnegative integers $a,b$ and  $\mathbf u=(u_1, \ldots, u_b)\in(\kappa^\ast)^b$, where $\kappa^\ast=\kappa\setminus\{0\}$.  Let  \begin{equation}\label{funcf} f(t)=t^{2a+\epsilon}\prod_{1\leq i\leq b}(t^2-u_i),\end{equation} where $ \epsilon\in \{0,1\}$. In \cite{CK}, Comes and Kujawa defined \begin{equation} \label{COBC defn}\OBC_\kappa^f=\AOBC_\kappa/I,\end{equation}
called the~\textsf{cyclotomic quotient} of $\AOBC_\kappa$ or the  \textsf{level $\ell$ oriented Brauer-Clifford supercategory}~\cite{CK},
where $\ell$ is the degree of $f(t)$ and $I$ is  the left  tensor ideal generated by $f(\xdot)$.

\begin{conj}\cite[Conjecture~7.1]{CK}\label{Cyclotomic basis conjecture} Suppose that $\kappa$ is a field of characteristic not two.   Given two words  $\ob a,\ob b\in\wrd$,  $\Hom_{\OBC_\kappa^f}(\ob a,\ob b)$ has basis given by  all equivalence classes of normally ordered dotted oriented Brauer-Clifford diagrams with bubbles of type $\ob a\to\ob b$ with fewer than $\ell$ $\fulldot$'s on each strand.
\end{conj}

Comes and Kujawa proved that Conjecture~\ref{Cyclotomic basis conjecture}
is true when
either $f(t) = t$ or $f(t) = t^2-u$ with $u\neq 0$. The proof for the second case is inspired by  \cite[\S4]{GRSS}.
Let $\AOBC_\kappa (0)$ (resp.,~$\OBC_\kappa^f(0)$) be  the supercategory obtained from $\AOBC_\kappa$ (resp.~ $\OBC_\kappa^f$) by imposing the relations  $\Delta_{k}=0$ for all $k>0$. Using certain representations of  finite $W$-superalgebras associated to queer Lie superalgebras  $\mathfrak q(n)$, Comes and Kujawa  are able to prove  Conjecture~\ref{Cyclotomic basis conjecture} for $\OBC_\kappa^f(0)$. In general, as far as we know, their conjecture remains open.

 The main result of  this paper is that Conjecture~\ref{Cyclotomic basis conjecture} holds over an arbitrary commutative ring $\kappa$ containing $2^{-1}$.  As a by-product, we prove that the cyclotomic walled Brauer-Clifford superalgebra defined by Comes and Kujawa in \cite{CK}  is isomorphic our cyclotomic walled Brauer-Clifford superalgebra in \cite[Definition~3.14]{GRSS} when $\kappa$ is an algebraically closed field with characteristic not two.

The contents of this paper are organized as follows. In section~2, we consider certain tensor modules in parabolic supercategory $\mathcal O$ for  $\mathfrak q(n)$ over $\mathbb C$. In section~3, we
prove that Conjecture~\ref{Cyclotomic basis conjecture} holds over an arbitrary commutative ring $\kappa$ containing $2^{-1}$. In section~4, we prove that the cyclotomic walled Brauer-Clifford superalgebra defined by Comes-Kujawa and ours are isomorphic if $\kappa$ is an algebraically closed field with characteristic not two.

\section{Schur-Weyl super-duality}

Let $\mfg$ be the queer Lie superalgebra  $\mathfrak{q}(n)$ of rank $n$ over $\C$. Then $\mfg$ has  a basis $e_{i,j}=E_{i,j}+E_{-i,-j}$ (even element), $f_{i,j}=E_{i,-j}+E_{-i,j}$ (odd element) for
$i,j\in I^+=\{1,2,...,n\},$ where $E_{i,j}$ is the $2n\times2n$ matrix with entry $1$ at $(i,j)$ position and zero elsewhere for $i,j\in I=I^+\cup I^-$, and  $I^-=-I^+$.

Let  $V=\C^{n|n}=V_{\bar0}\oplus V_{\bar1}$ be the
natural $\mfg$-supermodule (and the natural supermodule of the general linear Lie superalgebra $\mathfrak{gl}_{n|n}$) with basis $\{v_i\,|\,i\in I\}$.
Then $v_i$ has the parity $[v_i]=[i]\in\Z_2$, where $[i]=0$ and $[-i]= 1$ for $i\in I^+$.
Let $V^*$ be the linear dual  of $V$ with $\{\bar v_i\,|\,i\in I\}$ being its dual basis. Then  $V^*$ is a left $\mfg$-supermodule such that
\begin{eqnarray}\label{action-dual}
E_{a,b}\bar v_i=-(-1)^{[a]([a]+[b])}\d_{i,a}\bar v_b\mbox{ for }a,b,i\in I.\end{eqnarray}
Let $\fh=\fh_{\bar 0}\oplus\fh_{\bar 1}$ be a Cartan subalgebra of $\mfg$ with even part $\fh_{\bar0}={\rm span}\{h_i\,\,|\,i\in I^+\}$ and
odd part $\fh_{\bar 1}={\rm span}\{h_i':\,\,|\,i\in I^+\}$, where $h_i=e_{i,i}$ and $h_i'=f_{i,i}$ for all admissible $i$. Let $\fh^*_{\bar0}$ be the linear dual  of $\fh_{\bar0}$ with
 $\{\es_i\,|\,i\in I^+\}$ being the dual basis of $\{h_i\,\,|\,i\in I^+\}$. Then an element $\l\in\fh^*_{\bar0}$ (called a {\it weight}) can be written as \begin{equation}\label{weight-}\l=\SUM{i\in I^+}{}\l_i\es_i.\end{equation}
 Let $\mathfrak b$ be the standard Borel super subalgebra of $\mfg$ with even part $\mathfrak b_{\bar 0}={\rm span}\{e_{i,j}\,\,|\,i\leq j\in I^+\}$ and
odd part $\mathfrak b_{\bar 1}={\rm span}\{f_{i,j}\,\,|\,i\leq j\in I^+\}$.
Let $\mathcal O$ be the supercategory of all $\mfg$-supermodules $M$ such that:
\begin{enumerate}
\item $M$ is finitely generated as a $\mfg$-supermodule;
\item $M$ is locally finite-dimensional over $\mathfrak b$;
\item $M$ is semisimple over $\fh_{\bar 0}$.
\end{enumerate}
\noindent For any $\l\in \fh^*_{\bar0}$, let $I_\l$ be the irreducible $\fh$-supermodule. Then the dimension of $I_\l$  is  $2^{\lfloor\frac{\ell(\lambda)+1}{2}\rfloor}$ (see ~\cite{CW}), where $\ell(\l)$ is the number of non-zero parts of $\l$ and $\lfloor a\rfloor$ is the integer part of any nonnegative real number $a$.
Let $$M(\lambda):= \U(\mfg)\otimes _{\U(\mathfrak b)} I_\l$$ be the Verma supermodule with the highest weight $\lambda$, where in general  $\U(\mathfrak f)$ is the universal enveloping algebra of any Lie superalgebra $ \mathfrak f$. Then $M(\lambda)$ has the simple head denoted by  $L(\l)$. It is well known that   $L(\l)$ is of  finite dimensional if and only if $\l_i-\l_{i+1}\in\mathbb Z_{\geq0}$ and $\l_i-\l_{i+1}=0 $ implies that $\l_i=0$, for $1\leq i<n$.

Fix $ \epsilon\in\{0,1\}$ and two nonnegative integers $a$ and $b$. We define  $\mathbf n=(n_1,n_2, \ldots, n_{a+b+\epsilon})$ such that \begin{equation}\label{bign} n=\sum_{i=1}^{a+b+\epsilon} n_i,\end{equation} the summation of even positive integers $n_i$ for  $1\leq i\leq a+b+\epsilon$.
Let $\mathfrak p$
be the parabolic super subalgebra of $\mfg$ such that the Levi super subalgebra $\mathfrak l$ is  $ \oplus_{i=1}^{a+b+\epsilon} \mathfrak q(n_i)$.   Let $\mathcal O^{\mathfrak p}$ be the corresponding parabolic  supercategory $\mathcal O$. Then $\mathcal O^{\mathfrak p}$  is the full subcategory of $\mathcal O$
consisting of all $\mfg$-supermodules which  are locally finite-dimensional over $\mathfrak p$.
Throughout, we define
\begin{equation}\label{pii} p_0=0, \text{ and $p_i=\sum_{j=1}^i n_j$, and $\mathbf p_i=\{p_{i-1}+1, \ldots, p_i\}$ for $1\leq i\leq a+b+\epsilon$}\end{equation}  and let
\begin{equation}\label{lambda}
\Lambda=\{\l\in \fh^*_{\bar0}\mid  \l_j-\l_{j+1}\in\mathbb Z_{\geq 0} \text{ and $\l_j=0$ if $\l_j=\l_{j+1} $},\  p_i\leq j< p_{i+1}\},
\end{equation}
be the set of $\mathfrak l$-dominant weights.
For any $\l\in\Lambda$, the irreducible $\mathfrak l$-module $L(\l)^0$ with the highest weight $\l$ is  finite dimensional.
The parabolic Verma supermodule $M^{\mathfrak p}(\l)$  with the highest weight  $\l\in\Lambda$ is
$$M^{\mathfrak p}(\l):= \U(\mfg)\otimes _{\U(\mathfrak p)} L(\l)^0.$$
 For any $\mfg$-supermodule  $M$ and any  $r\!\in\!\Z^{\ge0}$, set $M^{r}= V^{\otimes r}\OTIMES M $. For convenience
we define the totally  ordered set \begin{equation}\label{ordered-set}J= J_1\cup\{0\}\
\mbox{ where $J_1=\{1,...,r\}$, } \end{equation}
such that
$r\prec r-1\prec\ldots\prec 1\prec 0$.
We write $M^{r}$ as
\begin{equation}\label{M-st==}M^{r}=\OT{i\in J}V_i,\mbox{  where $V_0=M$, $V_i=V$ if $i \in J_1$}. \end{equation}
 Hereafter all tensor products will be taken according to the total order $\prec$  on  $J$.  Then $M^r$ is a   left $\U(\mfg)^{\otimes(r+1)}$-supermodule such that the action is given by
 \begin{equation}\label{gact} \Big(\OT{i\in J} g_i\Big)\Big(\OT{i\in J} x_i\Big)=(-1)^{\sum\limits_{i\in J}{}[g_i]\sum\limits_{j\prec i}{}[x_j]}\OT{i\in J}(g_ix_i)\mbox{ for }g_i\in \U(\mfg),\ x_i\in V_i.\end{equation}
Via the coproduct of $\U(\mfg)$, it is a left $\U(\mfg)$-supermodule. In order to define the left action of $\AOBC_{\mathbb C}$ on $M^r$,
we define   \begin{eqnarray}\label{def-Omega}&\!\!\!\!\!\!\!\!\!\!\!\!\!\!\!\!\!\!\!\!\!\!\!\!\!\!\!&
\tilde e_{i,j}=E_{i,j}-E_{-i,-j},\ \ \  \tilde f_{i,j}=E_{-i,j}-E_{i,-j} \in {\mathfrak {gl}}_{n|n},\nonumber\\
&\!\!\!\!\!\!\!\!\!\!\!\!\!\!\!\!\!\!\!\!\!\!\!\!\!\!\!\!\!\!\!&
\Omega_1\!=\!\mbox{$\sum\limits_{i,j\in I^+}$}\tilde e_{i,j}\OTIMES  e_{j,i}\!-\!\mbox{$\sum\limits_{i,j\in I^+}$}\tilde f_{i,j}\OTIMES  f_{j,i}\in \mathfrak{gl}_{n|n} \OTIMES\mfg.\end{eqnarray}
 Let $c: V\to V$ be the odd linear map such that \begin{equation}\label{oddmapc}c(v_{i})=(-1)^{\p{v_{i}}}\sqrt{-1}v_{\bar{i}}, \ \ \text{for all $i \in I$.}\end{equation}
Since $\mathcal O$ and $\mathcal O ^{\mathfrak p}$ are closed under the functors $V\otimes -$ and $ V^*\otimes -$, we can use $\mathcal O^{\mathfrak p} $ (or $\mathcal O$) to replace the supercategory $\U(\mfg)$-smod  of left $\U(\mfg)$-supermodules in \cite[Theorem~4.4]{CK}.

\begin{Theorem}\cite[Theorem~4.4]{CK}\label{actofca}
There is a monoidal superfunctor $\Psi:\AOBC_{\mathbb C}\to\END(\mathcal O^{\mathfrak p})$  sending the objects $\up, \down$ to the endofunctors $V\otimes-, V^*\otimes-$, respectively, and  moreover,
\begin{align*}
    \Psi\left(\lcup\right)&:\operatorname{Id} \rightarrow V \otimes V^*\otimes-,
    \quad&
    m &\mapsto \sum_{i\in I} v_i\otimes v^*_i \otimes m,\\
    \Psi\left(\lcap\right)&:V^* \otimes V\otimes- \rightarrow \operatorname{Id},
    \quad&
    f\otimes v\otimes m  &\mapsto f(v) m,\\
    \Psi\left(\swap\right)&:V \otimes V \otimes- \rightarrow V \otimes V \otimes-,
    \quad&
    u \otimes v \otimes m &\mapsto (-1)^{\p{u}\p{v}} v\otimes u \otimes m,\\
    \Psi\left(\xdot\right)&:V \otimes - \rightarrow V \otimes -,
    \quad&
    v \otimes m &\mapsto \Omega_1(v \otimes m),\\
    \Psi\left(\cldot\right)&:V \otimes - \rightarrow V \otimes -,
    \quad&
    v \otimes m &\mapsto c(v)\otimes m.
\end{align*}
\end{Theorem}
Write $\Psi_M:\AOBC_{\mathbb C}\to \mathcal O$ (resp., $\mathcal O^{\mathfrak p}$) for the composition of $\Psi$ followed by evaluation at $M$ for any highest weight supermodule $M$ in $\mathcal O$ (resp., $\mathcal O^{\mathfrak p}$).
Given two $\mathfrak h$-supermodule (resp.,  $\mfg$-supermodule) $M$ and $N$, define $M\over N$ to be a supermodule which has a filtration of length two such that the top (resp., bottom) section is  isomorphic to $M$ (resp.,  $N$). Let  $\Pi$ be the parity change functor. The following result can be found   in the   proof of  \cite[Lemma~4.37]{Brundan}. Note that $V\cong \bigoplus_{i=1}^{n} I_{\varepsilon_i}$ as $\mathfrak h$-supermodules. Moreover, $I_{\varepsilon_i} $  has  basis $\{v_i,v_{-i}\}$, for $1\leq i\leq n$.

\begin{Lemma}\label{filtlam}(cf.  \cite[Lemma~4.37]{Brundan})
 Suppose that $\l\in \fh^*_{\bar0}$ such that  $\ell(\lambda)$ is even. As $\mathfrak h$-supermodules, there is an isomorphism   $V\otimes I_\l \cong \bigoplus_{i=1}^n I_i$, where
\begin{equation}\label{decp123}
I_i\cong\left\{
  \begin{array}{ll}
    I_{(\lambda+\varepsilon_i)}\oplus \Pi I_{(\lambda+\varepsilon_i)}, & \hbox{if $\l_i\not\in \{0,-1\}$;} \\
    I_{(\lambda+\varepsilon_i)}, & \hbox{if $\l_i= 0$;} \\
    I_{(\lambda+\varepsilon_i)}\over   I_{(\lambda+\varepsilon_i)}, & \hbox{if $\l_i= -1$.}
  \end{array}
\right.
\end{equation}
Moreover, $I_i$ has a basis $\{ v_i\otimes v, v_{-i}\otimes v\mid v\in S_\l\}$, where $S_\l$ is any basis of $I_\l$.
\end{Lemma}

For the simplification of notation, we use $x$ to denote $\Psi_{M(\l)}\left(\xdot\right) $ in the following result.
\begin{Lemma}\label{filml} Suppose that  $\l\in \fh^*_{\bar0}$ such that  $\ell(\lambda)$ is a nonnegative   even integer.
 Then
$V\otimes M(\l)$ has an $x $-stable $\U(\mfg)$-filtration
$$0=M_0 \subseteq M_1\subseteq \ldots \subseteq M_n =V\otimes M(\l) $$
such that \begin{equation}\label{vfiso}
M_i/M_{i-1}\cong\left\{
  \begin{array}{ll}
    M(\lambda+\varepsilon_i)\oplus \Pi M(\lambda+\varepsilon_i), & \hbox{if $\l_i\not\in \{0,-1\}$; } \\
    M(\lambda+\varepsilon_i), & \hbox{if $\l_i= 0$; } \\
    M(\lambda+\varepsilon_i)\over   M(\lambda+\varepsilon_i), & \hbox{if $\l_i= -1$.}
  \end{array}
\right.
\end{equation}
Moreover, $M_i/M_{i-1} $ is killed by
 $  x^2 -\lambda_i(\lambda_i+1)$ (resp., $ x$)  if $\lambda_i\neq 0$  (resp.,  if $ \lambda_i=0$).
\end{Lemma}

\begin{proof}  Recall $M(\l)=\U(\mfg)\otimes _{U(\mathfrak b)}I_{\l}$. As $\U(\mfg)$-modules, $$V\otimes  M(\l) \cong \U(\mfg)\otimes _{\U(\mathfrak b)}(V\otimes I_{\l}).$$ So the required filtration can be constructed such that
$M_i/M_{i-1}$ is generated by the images of  $\{v_i\otimes v, v_{-i}\otimes v\mid v\in S_\l\}$, where $S_\l$ is any basis of $I_\l$.
Moreover, \eqref{vfiso} follows from \eqref{decp123}.
The fact that  the filtration is   $x $-stable and
 $M_i/M_{i-1} $ is killed by  $  x^2 -\lambda_i(\lambda_i+1)$ (resp., $ x $) follows from  arguments in the proof of \cite[Lemma~3.2]{BD} for $\l_i\notin \frac{1}{2}\mathbb Z$ (see also \cite[Lemma~3.5]{BD1} for $\l_i\in \frac{1}{2}+\mathbb Z$).
However, their arguments are still available  if  $\l_i \in \mathbb Z$. The
 only difference appears when  $\l_i=0$.  We give a sketch of their arguments here.

 Let $\{m_i\mid 1\le i\le k\}$  be a basis of  the even subspace of $I_{\lambda}$. Recall that $h_i'=f_{ii}$ for all admissible $i$. If  $\l_i\neq 0$, then
 $\{h_i'm_j\mid 1\le j\le k\}$ is a basis of  the odd subspace of $I_{\l}$.
   Note that  $M_i/M_{i-1}$ is generated by the images of vectors in $A_i=
\{v_i\otimes m_j, v_{-i}\otimes m_j,v_i\otimes h_i'm_j, v_{-i}\otimes h_i'm_j\mid j=1,\ldots,k \} $.
  Suppose $\l_i=0$. If $\ell(\l)>0$, then we can find   a $t$ such that $\l_t\neq0$ and   $M_i/M_{i-1}$ is generated by the images of vectors in $A_i=
\{v_i\otimes m_j, v_{-i}\otimes m_j,v_i\otimes h_t'm_j, v_{-i}\otimes h_t'm_j\mid j=1,\ldots,k \}$. If $\ell(\lambda)=0$, then $I_\l=\mathbb C$ and $M_i/M_{i-1}$ is generated by the images of vectors in $A_i=
\{v_i\otimes 1, v_{-i}\otimes 1\}$ for $1\leq i\leq n$.
 For any  $v\in A_i$, we have  $(\tilde e_{r,s}\otimes e_{s,r}-\tilde f_{r,s}\otimes f_{s,r}) v=0$   unless $r\leq s=i$. If   $r<s=i$ then $ (\tilde e_{r,s}\otimes e_{s,r}-\tilde f_{r,s}\otimes f_{s,r}) v \in M_{i-1}$ for any $v\in A$.
 So, the filtration constructed above is $x$-stable, and $x$ acts on the highest weight space of $M_i/M_{i-1}$ via   $y_i:=\tilde e_{i,i}\otimes h_i-\tilde f_{i,i}\otimes h_i'$.
 By direct computation (see also the matrix of the endomorphism of $y_i$ with respect to  the highest weight space of $M_i/M_{i-1}$ in the proof of \cite[Lemma~3.2]{BD}), we have $[y_i^2-\l_i(\l_i+1)]v=0$ (resp., $y_iv=0$) for all  $v\in A_i$ if $\l_i\neq 0$ (resp., $\lambda_i=0$).
 Therefore,  $M_i/M_{i-1} $ is killed by   $  x^2 -\lambda_i(\lambda_i+1)$ (resp., $ x$)  if $\lambda_i\neq 0$  (resp.,   $ \lambda_i=0$)  as required.
\end{proof}

Hereafter, we fix  a weight  $\l\in\fh^*_{\bar0}$ such that
\begin{equation}\label{de of lam}
\l_{p_i+j}=\left\{
             \begin{array}{ll}
               l_i-j+1, & \hbox{$\text{ for } 1\leq j\leq n_{i+1}, 0\leq i\leq a+b-1$;} \\
0, & \hbox{$\text{ for } 1\leq j\leq n_{a+b+\epsilon},  i= a+b, \epsilon=1.$}
             \end{array}
           \right.
\end{equation}
where  $l_i=-1$ if $1\leq i\leq a $ and $l_i \notin \mathbb Z_{\geq 0}\cup\{-1\}$ if $ a+1\leq i\leq a+b$. Then $\l\in\Lambda$,  where $\Lambda$ is   the set of $\mathfrak l$-dominant weights defined in
 \eqref{lambda}.
 We identify $\l$ with $(\lambda_1, \lambda_2, \ldots, \l_n)$. Define $\l^{(i)}=(\lambda_{p_{i-1}+1},\ldots, \l_{p_i})$ for all $1\leq i\leq a+b+\epsilon$. Then $\l=(\l^{(1)},\ldots, \l^{(a+b)},\l^{(a+b+\epsilon)} )$.
Let $L(\l^{(i)})$ be the irreducible $\mathfrak q(n_i)$-module with the highest weight $\lambda^{(i)}$. Then  \begin{equation}\label{fds}L(\l)^0\cong \bigotimes_{i=1}^{a+b+\epsilon}L(\l^{(i)}),\end{equation}
 where $L(\l^{(a+b+\epsilon)})\cong\mathbb C$ if $\epsilon=1$ and  $ \l^{(i)}$ and $\l^{(i)}+\varepsilon_{p_{i-1}+1} $ are \emph{typical} as weights of $\mathfrak q(n_i)$ for $1\leq i\leq a+b$ (see \cite{Brundan}). Thanks to \eqref{decp123} and character considerations (cf. the finite dimensional typical character formula in \cite[Theorem~2]{Pe} or \cite[Theorem~4.8]{SZ}), we have  $V\otimes L(\l)^0\cong \bigoplus_{i=1}^{a+b+\epsilon} L_i$, where
\begin{equation}\label{filofl0}
L_i\cong \left\{
                      \begin{array}{ll}
                       L(\l+\varepsilon_{p_{i-1}+1})^0\over \Pi L(\l+\varepsilon_{p_{i-1}+1})^0, & \hbox{$1\leq i\leq a$;} \\
                       L(\l+\varepsilon_{p_{i-1}+1})^0\oplus \Pi L(\l+\varepsilon_{p_{i-1}+1})^0 , & \hbox{$a+1\leq i\leq a+b$;}\\
                        L(\l+\varepsilon_{p_{i-1}+1})^0, & \hbox{$i=a+b+\epsilon$ and $ \epsilon=1$.}
                      \end{array}
                    \right.
\end{equation}

In order to simplify the notation,  similar to the above we still use $x$ to denote
 $\Psi_{M^{\mathfrak p}(\l)} \left(\xdot\right)$ (the parabolic version of $x$) in the following result.
For $\epsilon=0$ and $a+b=1$, Theorem~\ref{zerop}(b) can be found in \cite[Lemma~4.4(a)]{GRSS}.
\begin{Theorem}\label{zerop} For any   $\l$  in \eqref{de of lam}, there is an   $x$-stable $\U(\mfg)$-filtration
$$0=M_0\subset M_1\subset \ldots \subset M_{a+b+\epsilon}=V\otimes M^{\mathfrak p}(\l)$$
 such that
 $$M_i/M_{i-1}\cong \left\{
                      \begin{array}{ll}
                        M^{\mathfrak p}(\l+\varepsilon_{p_{i-1}+1})\over \Pi M^{\mathfrak p}(\l+\varepsilon_{p_{i-1}+1}), & \hbox{$1\leq i\leq a$;} \\
                        M^{\mathfrak p}(\l+\varepsilon_{p_{i-1}+1})\oplus \Pi M^{\mathfrak p}(\l+\varepsilon_{p_{i-1}+1}) , & \hbox{$a+1\leq i\leq a+b$;}\\
                        M^{\mathfrak p}(\l+\varepsilon_{p_{i-1}+1}), & \hbox{$i=a+b+\epsilon$ and $ \epsilon=1$.}
                      \end{array}
                    \right.
 $$
Moreover,\begin{enumerate} \item  $M_i/M_{i-1} $ is killed by
 $ x^2-l_i(l_i+1)$ (resp., $  x $)  if $1\leq  i\leq a+b $  (resp.,  if $ i=a+b+\epsilon$ and $\epsilon=1$).
  \item $V\otimes M^{\mathfrak p}(\l)$ is killed by  $f(x)$
where $f(t)= t^{2a+\epsilon}\prod_{i=1}^b(t^2-u_i)$, such that $u_i=l_{a+i}(l_{a+i}+1)\neq 0$, for all $1\le i\le b$.
\item The superfunctor $\Psi_{M^{\mathfrak p}(\l)}$ in Theorem~\ref{actofca} factors through $\OBC_{\mathbb C}^f$, and thus induces a superfunctor
$\Psi_{M^{\mathfrak p}(\l)}^f: \OBC_{\mathbb C}^f\to\mathcal O^{\mathfrak p}$.\end{enumerate}

\end{Theorem}

\begin{proof} Recall that  $M^{\mathfrak p}(\l)=\U(\mfg)\otimes_{\U(\mathfrak b)}  L(\l)^0$. So, $V\otimes M^{\mathfrak p}(\l)\cong \U(\mfg)\otimes _{\U(\mathfrak p)}(V\otimes L(\l)^0)$.
Note that $L(\l)^0$ is the quotient of $M(\l)_{\mathfrak l}$, where $M(\l)_{\mathfrak l} $ is the Verma supermodule of $ \mathfrak l$ with the highest weight $\lambda$.
Let $\phi: V\otimes M^{\mathfrak p}(\l)\rightarrow V\otimes L(\l)^0$ be the epimorphism  induced by the canonical epimorphism from $M(\l)_{\mathfrak l}$ to $L(\l)^0$.
By Lemma~\ref{filml}, there is an $ \mathfrak l$-module  filtration of $V\otimes M(\l)_{\mathfrak l} $
$$0=M_{\mathfrak l, 0}\subset M_{\mathfrak l,1}\subset \ldots \subset M_{\mathfrak l,n}=V\otimes M(\l)_{\mathfrak l} $$
such that $M_{\mathfrak l, i}$ is generated by the images of  $A_i=\{v_j\otimes v , v_{-j}\otimes v\mid v\in S_\l, j\leq i\}$,
and $M_{\mathfrak l, i}/M_{\mathfrak l,i-1}$ is determined by \eqref{vfiso} as a filtration of Verma supermodules of $\mathfrak l$, where $S_\l$ is any basis of $I_\l$.
So, there is an $ \mathfrak l$-module filtration of $V\otimes L(\l)^0$
\begin{equation}\label{ffff}
0=N_0\subset N_1\subset \ldots \subset N_{n}=V\otimes L(\l)^0
\end{equation}
such that  $N_i=\phi(M_{\mathfrak l,i})$. Each  $N_i$ (resp., $N_{i}/N_{i-1}$) is the quotient of $M_{{\mathfrak l, i}}$ (resp., $M_{\mathfrak l, i}/M_{\mathfrak l,i-1}$) and   is generated  by the images of  $A_i$ (resp.,$A_i\setminus A_{i-1}$).
Note that $\l+\varepsilon_i\notin \Lambda$ if $i\neq p_{j-1}+1$ for any  $1\leq j\leq a+b+\epsilon$ and hence  $L(\l+\varepsilon_i)^0$ is infinite dimensional.
Suppose  $i\neq p_{j-1}+1$. If  $N_i\neq N_{i-1}$, by \eqref{vfiso}, $N_{i}/N_{i-1}$ must be infinite dimensional since it is a quotient of $M_{\mathfrak l, i}/M_{\mathfrak l,i-1}$ and $L(\l+\varepsilon_i)^0$ is infinite dimensional. This is a contradiction. Thus $N_i=N_{i-1}$ if $i\neq p_{j-1}+1$ for any  $1\leq j\leq a+b+\epsilon$.
So, the filtration in \eqref{ffff} can be reduced to
\begin{equation}\label{ffff2}
0=N_0\subset N_{1}\subset N_{p_1+1} \ldots \subset N_{p_{a+b+\epsilon-1}+1}=V\otimes L(\l)^0
\end{equation}
where $N_{p_{j}+1}/N_{p_{j-1}+1}$ is generated by images of  $\{v_{p_{j}+1 }\otimes v, v_{-(p_j+1)}\otimes v\mid v\in S_\l\}$.
By \eqref{filofl0} and character consideration, we have
\begin{equation}\label{ndnns}
N_{p_{j}+1}/N_{p_{j-1}+1}\cong L_j, \text{ for } 1\leq j\leq a+b+\epsilon-1, \text{ and }N_1\cong L_1
\end{equation}
where $ L_i$ is given in \eqref{filofl0}.
 Now define  $M_j:=\U(\mfg)\otimes _{U(\mathfrak p)} N_{p_{j-1}+1}$ for $1\leq j\leq a+b+\epsilon$. Then the required filtration of $V\otimes M^{\mathfrak p}(\l)$ follows form \eqref{ffff2}--\eqref{ndnns} and (a) follows from the proof of  Lemma~\ref{filml}.
Via (a), we immediately have   (b).
 Finally, (c) follows from (b) and Theorem~\ref{actofca}.\end{proof}

\begin{rem} For any $0\neq u_i \in \mathbb C$, there exists $\ell_{a+i}\notin \mathbb Z_{\geq 0}\cup\{-1\}$  such that $u_i=\ell_{a+i}(\ell_{a+i}+1)$.
This enables us to choose an $\mathfrak l$--dominant weight $\l$ in \eqref{de of lam} such that, for any  $ f(t)= t^{2a+\epsilon}\prod_{i=1}^b(t^2-u_i)\in \mathbb C[t]$,  there is a superfunctor $\Psi_{M^{\mathfrak p}(\l)}^f:  \OBC_{\mathbb C}^f\to\mathcal O^{\mathfrak p}$  .
\end{rem}

For $1\le m$ and $1\le i, j\le n$, Sergeev~\cite{Ser1} defined $x^{\bar 0}_{i,j}(1)=e_{i,j}$, $ x^{\bar 1}_{i,j}(1)=f_{i,j}$, and
\begin{equation}\label{def of x0}
x^{\bar 0}_{i,j}(m)=\sum_{s=1}^n(e_{i,s}x^{\bar 0}_{s, j}(m-1)+(-1)^{m-1} f_{i,s}x^{\bar1}_{s,j}(m-1) ),
\end{equation}
\begin{equation}\label{def of x1}
x^{\bar 1}_{i,j}(m)=\sum_{s=1}^n(e_{i,s}x^{\bar 1}_{s, j}(m-1)+(-1)^{m-1} f_{i,s}x^{\bar0}_{s,j}(m-1) ),
\end{equation} for $m>1$ and proved
 the  following  relations:

\begin{equation}\label{cmmurel}
\begin{aligned}
&[e_{i,j}, x_{s,t}^{\bar 0}(m)]= \delta_{j,s}x_{i,t}^{\bar 0}(m)-\delta_{i,t}x_{s,j}^{\bar 0}(m),\\
&[f_{i,j}, x_{s,t}^{\bar 0}(m)]= (-1)^{m-1}\delta_{j,s}x_{i,t}^{\bar 1}(m)-\delta_{i,t}x_{s,j}^{\bar 1}(m),\\
&[e_{i,j}, x_{s,t}^{\bar 1}(m)]= \delta_{j,s}x_{i,t}^{\bar 1}(m)-\delta_{i,t}x_{s,j}^{\bar 1}(m),\\
&[f_{i,j}, x_{s,t}^{\bar 1}(m)]= (-1)^{m-1}\delta_{j,s}x_{i,t}^{\bar 0}(m)+\delta_{i,t}x_{s,j}^{\bar 0}(m).\\
\end{aligned}
\end{equation}
Sergeev~\cite{Ser1} defined the following central elements in $\U(\mfg)$:
\begin{equation}\label{defofsr}
S_r:=\sum_{i=1}^nx_{i,i}^{\bar 0}(2r-1) \ \ \text{ for $r\in \mathbb Z_{>0}$.}
\end{equation}
 By \cite[Theorem~4.5,~Propostion~4.6 ]{CK}, we immediately have the following  result.
\begin{Prop}\label{actiononm}
For any positive integer $r$, $\Psi_{M^{\mathfrak p}(\l)}^f(\Delta_{2r-1}) u=-2\sigma(S_r)u$ for all  $u\in M^{\mathfrak p}(\lambda) $, where   $\sigma: \U(\mfg)\rightarrow \U(\mfg)$ is the antipode such that $\sigma(g)=-g$, for any $g\in \mfg$.
\end{Prop}

\begin{Defn}\label{zrlam} Given a  $\lambda$  in \eqref{de of lam} and a positive integer $r$, let
$$ z_r(\lambda)= -\sum 2^{s-1}\prod_{j=1}^s \lambda_{i_j} (\lambda_{i_j}^2+\lambda_{i_j})^{a_j}  ,
$$
where the summation is  over all $1\leq s\leq r, 1\leq i_1<i_2<\ldots <i_s\leq n$, and $a_1,a_2,\ldots,a_s\in \mathbb N$
such that  $\sum_{j=1}^s a_j=r-s$.\end{Defn}
 The above element $z_r(\lambda)$ can be obtained from the element defined in  \cite[Lemma~8.4]{BK}  by replacing  $\l_i$ with $-\l_i$ for all $1\leq i\leq n$.
 The following two results and their proofs are essentially the same as   \cite[Lemma~8.4]{BK}. The difference is that we compute the actions of $\sigma(S_r)$ on any parabolic Verma supermodule, while they compute the actions of $S_r$. However,   one  can not directly get the following two  results from  \cite[Lemma~8.4]{BK}. Note that $\sigma$ is a superalgebra anti-involution.
\begin{Lemma}\label{claim} Suppose
$1\le m$ and $ 1\leq i<j\leq n$. Then
$
\sigma(x^{\bar0}_{i,j}(m))\equiv\sigma(x^{\bar1}_{i,j}(m))\equiv0 \pmod J$,
where   $J$ is  the left superideal of $\U(\mfg)$ generated by all $e_{k,l}$ and $f_{k,l}$ such that  $k<l$.\end{Lemma}
\begin{proof}Obviously, $
\sigma(x^{\bar0}_{i,j}(1))\equiv\sigma(x^{\bar1}_{i,j}(1))\equiv0 \pmod J$.  In general, by \eqref{def of x0}, \eqref{cmmurel} and  inductive assumption,    we have
$$
\begin{aligned}
\sigma(x^{\bar0}_{i,j}(m))=&\sum_{s=1}^n(-\sigma(x^{\bar0}_{s,j}(m-1))e_{i,s}+(-1)^{m-1}\sigma(x^{\bar1}_{s,j}(m-1))f_{i,s})\\
\equiv &  \sum_{s=1}^i(-\sigma(x^{\bar0}_{s,j}(m-1))e_{i,s}+(-1)^{m-1}\sigma(x^{\bar1}_{s,j}(m-1))f_{i,s})\pmod J\\
 \equiv & 0\pmod J.
\end{aligned}$$
Similarly,  one can  verify   $\sigma( x^{\bar1}_{i,j}(m))\equiv 0 \pmod J$.\end{proof}
\begin{Prop}\label{actionofsr}
Let $u$ be a highest weight vector of $M^{p}(\lambda)$. Then \begin{enumerate}\item[(1)]  $ \sigma(S_r) u= z_r(\lambda)u$, \item [(2)] $\Psi_{M^{\mathfrak p}(\l)}^f(\Delta_{2r-1})u=-2z_r(\lambda)u$ for any positive integer $r$.\end{enumerate}
\end{Prop}

\begin{proof} Suppose $m$ is an odd positive integer. By Lemma~\ref{claim} and \eqref{cmmurel}, we have
\begin{equation}\label{SS1}
\begin{aligned} &
\sigma(x^{\bar0}_{i,i}(m))=\sum_{s=1}^n(-\sigma(x^{\bar0}_{s,i}(m-1))e_{i,s}+\sigma(x^{\bar1}_{s,i}(m-1))f_{i,s})\\
\equiv & -\sigma(x^{\bar0}_{i,i}(m-1))h_i+\sigma(x^{\bar1}_{i,i}(m-1))h'_i+ \sum_{s=1}^{i-1}(-\sigma(x^{\bar0}_{s,i}(m-1))e_{i,s}+\sigma(x^{\bar1}_{s,i}(m-1))f_{i,s})\pmod J\\
\equiv &-\sigma(x^{\bar0}_{i,i}(m-1))h_i+\sigma(x^{\bar1}_{i,i}(m-1))h'_i \pmod J.
\end{aligned}\end{equation}
Similarly, one can verify
\begin{equation}\label{SS2} \begin{aligned}
\sigma(x^{\bar0}_{i,i}(m-1))\equiv &-\sigma(x^{\bar0}_{i,i}(m-2))h_i-\sigma(x^{\bar1}_{i,i}(m-2))h'_i-2\sum_{s=1}^{i-1}\sigma (x^{\bar0}_{s,s}(m-2)) \pmod J\\
\sigma(x^{\bar1}_{i,i}(m-1))\equiv &-\sigma(x^{\bar1}_{i,i}(m-2))h_i+\sigma(x^{\bar0}_{i,i}(m-2))h'_i\pmod J.
 \end{aligned}\end{equation}
Combining \eqref{SS1}-\eqref{SS2} yields
\begin{equation}\label{sigmax1}
\sigma(x^{\bar0}_{i,i}(m))\equiv\sigma(x^{\bar0}_{i,i}(m-2))(h^2_i+h_i)+2\sum_{s=1}^{i-1}\sigma (x^{\bar0}_{s,s}(m-2))h_i\pmod J.
\end{equation}
By \eqref{sigmax1} and inductive assumption on $r$, we have
\begin{equation}\label{sigmax2}
\sigma(x^{\bar0}_{i,i}(2r-1))\equiv-\sum 2^{s-1}h_{i_1}h_{i_2}\cdots h_{i_s} y_{i_1}^{a_1}y_{i_2}^{a_2}\cdots
y_{i_s}^{a_s}\pmod J,
\end{equation}
where $y_i:=h_i^2+h_i$ and the summation  is over all $1\leq s\leq r$, $1\leq i_1<i_2<\ldots<i_s=i$, and $a_1,a_2,\ldots,
a_s\in \mathbb N$ such that  $a_1+a_2+\ldots+a_s=r-s$. Since $x u=0$ for all $x\in J$,  by \eqref{defofsr} and \eqref{sigmax2}, we have  $ \sigma(S_r) u= z_r(\lambda)u$, proving (1). Finally, (2) follows from (1) and Propositions~\ref{actiononm}.
\end{proof}

\section{Proof of Conjecture~\ref{Cyclotomic basis conjecture}}

The aim of this section is to give a proof of Conjecture~\ref{Cyclotomic basis conjecture} over an arbitrary commutative ring $\kappa$  containing $2^{-1}$. First, we assume $\kappa=\mathbb C$.
Recall that $\lambda$ is an $\mathfrak l$-dominant weight given in \eqref{de of lam}. Hereafter, we fix a   even  highest weight vector $v_\l$ of $L(\l)^0$.
We always assume  $n_i\geq 2r$, for $1\leq i\leq a+b+\epsilon$.

\begin{Defn}
For $1\leq i\leq a+b$, let
\begin{equation}\label{special monomials} B_{i}=\left\{
(f_{p_{i}, p_i-r})^{\beta_r} (f_{p_{i}-1, p_i-r-1})^{\beta_{r-1}}\ldots (f_{p_{i}-r+1, p_i-2r+1})^{\beta_{1}} \mid  \beta_j\in\{0,1\}, 1\leq j\leq r\right\}.
\end{equation}
\end{Defn}
\begin{Lemma}\label{basisoflevi}There is a subset $B$ of  the PBW monomial basis, which   contains all monomials
$\{b_{a+b}b_{a+b-1}\cdots b_{1}\mid    b_i\in B_i, 1\leq i\leq a+b\}$ such that $\{bv_\l~|~b\in B\}$ is a basis of $L(\l)^0$.
\end{Lemma}
 \begin{proof} By \cite[(4.9)]{GRSS}, we have the  result for $L(\l^{(i)})$,   $1\leq i\leq a+b$. Thanks to \eqref{fds}, we have the result in general.
\end{proof}
Recall
 $\mathbf p_i$ in \eqref{pii} for all $1\le i\le a+b+\epsilon$.
  Let $ \mathfrak u^{\mathfrak l}$ (resp. $ \mathfrak u^{\mathfrak l,-}$) be the nilradical (resp., opposite nilradical) of $\mathfrak p$. Then $B^{\bar 0}_\mathfrak l$
(resp., $B^{\bar1}_\mathfrak l$) is a basis of even subspace $ \mathfrak u^{\mathfrak l,-}_{\bar 0}$ (resp., odd subspace     $ \mathfrak u^{\mathfrak l,-}_{\bar 1}$) of
  $ \mathfrak u^{\mathfrak l,-}$,  where
\begin{equation}\label{nilb}\begin{aligned}   B^{\bar 0}_\mathfrak l& =\left\{e_{i,j}\mid i>j, \{i,j\}\not\subset \mathbf p_k, 1\leq k\leq a+b+\epsilon\right\},\\
B^{\bar1}_\mathfrak l&=\left\{f_{i,j}\mid i>j, (i,j)\not\subset \mathbf p_k, 1\leq k\leq a+b+\epsilon\right\}.\\
\end{aligned}
\end{equation}
It is known that  the symmetric power  $S(\mathfrak u^{\mathfrak l,-} )\cong S(\mathfrak u^{\mathfrak l,-}_{\bar 0} )\otimes \bigwedge \mathfrak u^{\mathfrak l,-}_{\bar 1}$, where $\bigwedge \mathfrak u^{\mathfrak l,-}_{\bar 1}$ is the usual
exterior power. Moreover, $S(\mathfrak u^{\mathfrak l,-} )$ has  basis
\begin{equation}
B_\mathfrak u= \left\{ \prod_k(e_{i_k,j_k})^{\delta_k}\prod_m(f_{i_m,j_m})^{\sigma_m}\mid \delta_k\in\mathbb Z_{\geq 0}, \sigma_m \in\{0,1\}\right \},
\end{equation}
where the first product (resp.,  the second product) is taken over any fixed order (for example, the lexicographic order) on  $B^{\bar 0}_\mathfrak l$
(resp. $B^{\bar1}_\mathfrak l$).

\begin{Cor}\label{basisofm}Let  $M$ be the parabolic Verma supermodule $M^{\mathfrak p} (\l)$, where $\lambda$ is given in \eqref{de of lam}. Then \begin{enumerate}\item[(1)]
$M$ has basis $\{z v_\l \mid z\in B_M\}$ where $B_M=\{yb\mid y\in  B_\mathfrak u, b\in B \}$, where $B$ is given in Lemma~\ref{basisoflevi}.
 \item[(2)]   $V^{\otimes r}\otimes M$ has basis
$B_{M, r}=\{ v_\mathbf i\otimes uv_\l\mid \mathbf i\in I^r, u\in B_M\}
$.\end{enumerate}
\end{Cor} \begin{proof} Thanks to Lemma~\ref{basisoflevi}, we immediately have (1)-(2).  \end{proof}

Let  $\U(\mfg)^{-}$ be  the negative part of $\U(\mfg)$. For any $f_{i,j}$, $e_{i,j}$ $\in \U(\mfg)^{-}$, we define $\text{deg}(f_{i,j})$ $=$ $\text{deg}(e_{i,j})$ $=1$. This gives a $\mathbb Z$-grading on $\U(\mfg)^{-}$.
If  $x\in \U(\mfg)^{-}$ is a PBW monomial,  then $\text{deg}(x)$  is equal to the numbers of $f_{i,j}$'s and $e_{i,j}$'s appearing in the product of $x$. For any basis element $yb v_\l\in B_{M, r}$, we    say $yb v_\l$ is of degree $\text{deg}(yb)$.

\begin{Defn}\label{vbeta} Suppose $\beta=(\beta_r, \beta_{r-1}, \ldots, \beta_1)\in \underline \ell^r$, where  $\underline\ell=\{0,1,\ldots, \ell-1\}$.
 Define $v^\beta=v_{i_{r, \beta_r}}\otimes v_{i_{r-1,\beta_{r-1}}}\otimes \ldots \otimes v_{i_{1,\beta_{1}}}$ and $y^\beta=y_{r,\beta_r}y_{r-1,\beta_{r-1}}\cdots y_{1,
\beta_1}$, where \begin{enumerate}\item $y_{k,\beta_k}$ is the ordered product $(\prod_{j=1}^{\beta_k-1} f_{i_{k,j}, i_{k,j+1}}) f_{ i_{k,0}, i_{k,1}}$, $1\le k\le r$, \item   $i_{k,0}=n-k+1$  and
$
i_{k, j}=  -( i_{k, j-1}-\gamma_j)$, if $1\le j\le \beta_k$, and $1\le k\le r$, \item   $\gamma_j=\frac{1-(-1)^{j+\epsilon}}{2}r+\frac{1+(-1)^{j+\epsilon}}{2}n_{a+b+ \epsilon-\lfloor \frac{j+\epsilon-1}{2}\rfloor}$,  if   $1\leq j\leq \beta_k$ and  $1\leq k\leq r$.\end{enumerate}\end{Defn}

Write $w^\beta= v^\beta\otimes x^{\beta}v_\lambda$,
such that $x^\beta\in B_M$ can be  obtained from $y^\beta$ by changing the order of its factors.
We  say $w^\beta$ is of degree $\text{deg}(x^\beta)$. Obviously, $\text{deg}(x^\beta)=\text{deg}(y^\beta)$.
We define a total order on $\underline \ell^r$ such that for any $\beta, \beta'\in \underline \ell^r$,  $\beta'<\beta$ if  either $\sum_{i}\beta_i'<\sum_{i}\beta_i$, or $\sum_{i}\beta_i'=\sum_{i}\beta_i$ and the leftmost nonzero entry of $\beta-\beta'$ is positive.

Following \cite{CK}, we count a strand of a dotted oriented Brauer-Clifford diagram  from  right to left. Let $x_k$ be obtained from the identity diagram by placing a single $\fulldot$ on the $k$th strand.   For example,
\begin{equation}\label{xlabel}
    x_2=~
    \begin{tikzpicture}[baseline = 7.5pt, scale=0.75, color=\clr]
        %\draw[->,thick] (0,0) to (0,1);
        %\draw[->,thick] (0.5,0) to (0.5,1);
        \draw[->,thick] (1,0) to (1,1);
        \draw[->,thick] (1.5,0) to (1.5,1);
        \draw[->,thick] (2,0) to (2,1);
        \draw[->,thick] (2.5,0) to (2.5,1);
        \draw[->,thick] (3,0) to (3,1);
        \draw[->,thick] (3.5,0) to (3.5,1);
        \draw (3,0.4) \bdot;
    \end{tikzpicture}
\end{equation} if $r=6$.
Let $\mathbf 0=(0,0,\ldots,0)\in \underline \ell^r$.
We define  $c_{\beta,\beta'}$ to be the  coefficient of $w^{\beta}$ when $\Psi_M^f(x_r^{\beta_r'}\cdots x_1^{\beta_1'})(w^{\mathbf{0}})$ is written in terms of the basis for $V^{\otimes r}\otimes M$ in Corollary~\ref{basisofm}.

Recall that  $\Omega_1$ is in \eqref{def-Omega}. Let  $i_{k,j}$ and $ i_{k,j+1}$ be  given in Definition~\ref{vbeta}.

\begin{Lemma}\label{adlee}  Expressing $\Omega_1(v_{i_{k,j}}\otimes v_\l)$ as
  a linear combination of basis elements  in $B_{M,1}$ (see Corollary~\ref{basisofm}), we see that  the highest degree of its terms is  $1$.
Moreover,
\begin{enumerate}
 \item[(1)]   $ v_{i_{k,j+1}}\otimes f_{i_{k,j},i_{k+1,j}}v_\l$ is a term of  $\Omega_1(v_{l}\otimes v_\l)$ if and only if  $l=i_{k,j}$.
\item[(2)] When $l=i_{k,j}$,  the coefficient  of $ v_{i_{k,j+1}}\otimes f_{i_{k,j},i_{k+1,j}}v_\l$ is $\pm1$.

\end{enumerate}
\end{Lemma}

\begin{proof}Straightforward computation. See the definition of $\Omega_1$  in \eqref{def-Omega}.
\end{proof}

\begin{Lemma}\label{L:coefficients technical lemma}
 Suppose $\beta,\beta'\in\underline \ell^r$. We have  $c_{\beta,\beta}=\pm 1$ and $c_{\beta,\beta'}=0$ if
  $\beta'<\beta$.
\end{Lemma}

\begin{proof}
By  Theorem~\ref{actofca}, $x_k$ acts on $V^{\otimes r}\otimes M$ via  $\Psi_M^f(x_k)=1^{\otimes r-k}\otimes \Omega_1|_{V, V^{\otimes k-1}\otimes M}$. In other words, when we compute  $\Psi_M^f(x_k)$, we consider  $V$ (resp., $V^{\otimes k-1}\otimes M$)  as the first (resp., second) tensor factor.
 Since we are going to consider the  terms
of $\Psi_M^f(x_r^{\beta_r'}\cdots x_1^{\beta_1'})(w^{\mathbf{0}})$ with the highest degree, $x_k$ can be replaced by   $\pi_{k,0}(\Omega_1)$,
where $\pi_{k,0}:\U(\mfg)^{\otimes 2}\rightarrow \U(\mfg)^{\otimes(k+1)}$ is the linear map such that $\pi_{k,0}(g_1\otimes g_2)=g_1\otimes 1^{\otimes k-1}\otimes g_2$, for all $g_1,g_2\in \U(\mfg)$.
By  Lemma~\ref{adlee}(2) and  induction on $|\beta|:=\sum_i\beta_i$, we have $c_{\beta,\beta}=\pm 1$.

 Since the degree of  any   term  in the expression of  $\Psi_M^f(x_r^{\beta_r'}\cdots x_1^{\beta_1'})(w^{\mathbf{0}})$
is less than  $\sum_{i}\beta_i'$, we have   $c_{\beta,\beta'}=0$ if  $\sum_{i}\beta_i'<\sum_{i}\beta_i$.  Now, we assume  $\sum_{i}\beta_i'=\sum_{i}\beta_i$ and $\beta'<\beta$. We prove $c_{\beta,\beta'}=0$  by  induction on $\sum_{i}\beta_i'$.
Obviously, $c_{\beta,\beta'}=0$ if  $\sum_{i}\beta_i'=0$. Otherwise,
let $k$ be the maximal integer  such that $ \beta_k>0$.

 \textbf{Case 1:} $\beta_k'=0$. Suppose that $v_\mathbf i\otimes uv_\l \in B_{M,r}$, which  appears as a term of
 $\Psi_M^f(x_r^{\beta_r'}\cdots x_1^{\beta_1'})(w^{\mathbf{0}})$ with the highest degree. Since we are assuming that $\beta_k'=0$,   $v_{i_k}$ (i.e. $k$th component of $v_\mathbf i$) must be  $v_{n-k+1}$. Further, $v_{i_k}$ is the $k$th component of $v^{\mathbf 0}$ (see the definition of  $v^{\mathbf 0}$ in Definition~\ref{vbeta}). By the definition of
  $i_{k, j}$ in Definition~\ref{vbeta}, we have   $c_{\beta,\beta'}=0$.

 \textbf{Case 2:} $\beta_k'>0$. If   $c_{\beta,\beta'}\neq0$, we define   $\bar \beta=\beta-(0^{r-k},1,0^{k-1}) $ and $\bar\beta'=\beta'-(0^{r-k},1,0^{k-1})$.
By  Lemma~\ref{adlee}(1),
 $ w^{\bar \beta}$ is  a term in the expression of   $\Psi_M^f(x_r^{\beta_r'}\cdots x_{k+1}^{\beta'_{k+1}}x_k^{\beta'_k-1} x_{k-1}^{\beta'_{k-1}}\cdots  x_1^{\beta_1'})(w^{\mathbf{0}})$ with non-zero coefficient. So,
 $c_{\bar\beta,\bar\beta'}\neq0$.
 This contradicts  our inductive assumption since  $\sum_{i}\bar\beta_i'=\sum_{i}\bar\beta_i<\sum_{i}\beta_i'$ and $\bar\beta'<\bar\beta$ and $|\bar \beta|=|\beta|-1$.
Thus, $c_{\beta,\beta'}=0$ as required.
\end{proof}

For any  dotted oriented Brauer-Clifford diagram $d$ of type  $\up^r\to\up^r$,    Comes and Kujawa  defined the oriented Brauer-Clifford diagram $\undot{d}$, which is obtained from $d$ by removing all $\fulldot$'s. For example, if
\begin{equation}\label{undot example}
    d=~
    \begin{tikzpicture}[baseline = 10pt, scale=0.5, color=\clr]
        \draw[->,thick] (0,0) to[out=up, in=down] (2,2);
        \draw[->,thick] (1,0) to[out=up, in=down] (4,2);
        \draw[->,thick] (2,0) to[out=up, in=down] (0,2);
        \draw[->,thick] (3,0) to[out=up, in=down] (1,2);
        \draw[->,thick] (4,0) to[out=up, in=down] (5,2);
        \draw[->,thick] (5,0) to[out=up, in=down] (3,2);
        \draw (0.15,1.5) \wdot;
        %\draw (1.85,1.5) \wdot;
        \draw (4.9,1.5) \wdot;
        %\draw (0.15,0.45) \bdot;
        \draw (2.85,0.45) \bdot;
        \draw (2.85,1.05) \bdot;
        \draw (3.85,1.05) \bdot;
        %\draw (4.85,0.45) \bdot;
    \end{tikzpicture},
    \quad\text{then}~
    \undot{d}=~
    \begin{tikzpicture}[baseline = 10pt, scale=0.5, color=\clr]
        \draw[->,thick] (0,0) to[out=up, in=down] (2,2);
        \draw[->,thick] (1,0) to[out=up, in=down] (4,2);
        \draw[->,thick] (2,0) to[out=up, in=down] (0,2);
        \draw[->,thick] (3,0) to[out=up, in=down] (1,2);
        \draw[->,thick] (4,0) to[out=up, in=down] (5,2);
        \draw[->,thick] (5,0) to[out=up, in=down] (3,2);
        \draw (0.15,1.5) \wdot;
      %  \draw (1.85,1.5) \wdot;
\draw (4.9,1.5) \wdot;
    \end{tikzpicture}
    ~.
\end{equation}
 Suppose $d$ is a normally ordered dotted oriented Brauer-Clifford diagram  without bubbles of type  $\up^r\to\up^r$. Let  $\beta_k(d)$  be the number of $\fulldot$'s on the $k$th strand of $d$. Then

\begin{equation}\label{d}
    d=\undot{d}\circ x_r^{\beta_r(d)}\circ\cdots\circ x_1^{\beta_1(d)}.
\end{equation}
(see \cite[\S5]{CK}).

\begin{Prop}\label{basis theorem for level l Hecke-Clifford} Suppose  $r$ is a positive  integer. Let $\OBC_{\mathbb C}^{ f}(\delta)$ be the cyclotomic oriented Brauer-Clifford supercategory over $\mathbb C$ with respect to  $ f(t)=t^{2a+\epsilon}\prod_{i=1}^b(t^2- u_i)\in\mathbb C[t]$ such that $\delta_{2k}=0$ and $\delta_{2k-1}=-2z_k(\lambda)$, where $z_k(\l)$ is given in Definition~\ref{zrlam} for any positive integer $k$.
Let   $S$   be the set of  all elements $
\undot{d}\circ x_r^{\beta_r(d)}\circ\cdots\circ x_1^{\beta_1(d)}$,
where \begin{enumerate} \item $d$ ranges over all equivalence classes of  normally ordered dotted oriented Brauer-Clifford diagrams without bubbles of type  $\up^r\to\up^r$,\item
 $(\beta_r(d), \beta_{r-1}(d), \ldots, \beta_1(d))\in \underline \ell^r$, where $\underline \ell=\{0, 1, \ldots, \ell-1\}$.\end{enumerate}
\noindent For a finite subset $A\subset  S$, we have  $\sum_{d\in A} p_d d=0$  if and only if    $p_d=0$ for all $d\in A$.
 \end{Prop}
\begin{proof}
Since  $\undot{d}:\up^r\to\up^r$ can be  decomposed in terms of  $\swap$'s and $\cldot$'s, we see that $\undot{d}$  is invertible. Suppose $\sum_{d\in A} p_d d=0$. If there is a $d\in A$ such that $p_d\neq 0$, we can find
a  $d_0\in A$  such that $p_{d_0}\not=0$ and $\beta(d_0)\geq \beta(d)$ for all $d\in A$ with $p_{d} \not=0$. Thanks to Lemma~\ref{L:coefficients technical lemma},
  the coefficient  of $w^{\beta(d_0)}$ in $\Psi_M^f(\undot{d_0}^{-1}\circ \sum_{d\in A}p_{d} d)(w^{\mathbf{0}})$ is $p_{d_0}c_{\beta(d_0),\beta(d_0)}$, which is non-zero. So,  $\Psi_M^f(\sum_{d\in A}p_{d} d)\not=0$, a contradiction.
\end{proof}

Now, we consider $z_k(\lambda)$ as a polynomial in variables  $n_1,n_2, \ldots, n_{a+b}$ with coefficients in $\mathbb C$.
Given a weight $\l$ in  \eqref{de of lam}, we consider  the morphism
\begin{equation}\label{definition of varphi}
\psi: \mathbb C^{a+b}\rightarrow \mathbb C^{a+b}, \quad (n_1,n_2,\ldots, n_{a+b}) \mapsto(z_1(\lambda),z_2(\lambda), \ldots, z_{a+b}(\lambda)).
\end{equation}

\begin{Lemma}\label{dominant}Let $\psi$ be the morphism   in \eqref{definition of varphi}. Then $\psi$  is dominant over $\mathbb C$.

\end{Lemma}
\begin{proof} We view  $n_1,n_2, \ldots, n_{a+b}$ as variables. So,
it  is enough to show that the determinant $\det J_\varphi$ is non zero, where   $J_\psi:= (\frac{\partial z_k(\lambda)}{\partial n_s} )_{1\leq k,s\leq a+b}$ is the   Jacobian matrix.

We claim that the highest degree term  of $\frac{\partial z_k(\lambda)}{\partial n_s}$  is $a_k n_s^{2k-1}$, for some $a_k\neq 0$ such that $a_k$ is independent of $s$ whenever $s\neq k$.
If so,  then, up to some non-zero scalar,  the term of $\det J_\psi$ with  the highest degree forms the following determinant
$$ \text{det}\left(
     \begin{array}{cccc}
       n_1 & n_2 & \ldots & n_{a+b} \\
       n_1^3 & n_2^3 & \ldots & n_{a+b}^3 \\
       \vdots & \vdots & \ddots& \vdots\\
       n_1^{2(a+b)-1} & n_2^{2(a+b)-1} & \ldots & n_{a+b}^{2(a+b)-1} \\
     \end{array}
   \right),
$$
which is non-zero.
So, it remains to prove the claim.
By \eqref{zrlam}, the highest degree term  of $\frac{\partial z_k(\lambda)}{\partial n_s}$  only appears in  some term of $\frac{-\partial\sum_{i=p_{s-1}+1}^{p_s}\lambda_i(\lambda_i^2+\lambda_i)^{k-1}}{\partial n_s}$. Therefore, it
 appears in  some term  of $\frac{-\partial g_{k,s}}{\partial n_s}$, where
 $g_{k,s}:=\sum_{i=p_{s-1}+1}^{p_s}\lambda_i^{2k-1}$. Thanks to  \eqref{de of lam}, we  write
$$g_{k,s}= \sum_{j=1}^{n_s} (l_s-n_s+j)^{2k-1}.$$
This shows that the highest degree term of $g_{k,s}$  is $b_k n_s^{2k}$, where $b_k\in\mathbb C^*$, a non-zero scalar which  is independent of $s$. However, when  we use \eqref{de of lam} to obtain $g_{k,s}$,  $ n_s $ ranges over all  big enough even integers. Since $g_{k,s}$ is a polynomial in variable  $n_s$ with coefficients in $\mathbb C$,  it is  available for all  $n_s\in \mathbb C$. So, the  term  of $\frac{\partial g_{k,s}}{\partial n_s}$ with  highest degree is $ 2k b_k n_s^{2k-1}$. This proves our claim  by setting  $a_k=-2kb_k$.
\end{proof}

\begin {Cor}\label{zeropoly}As $(n_1, n_2,\ldots, n_{a+b})$ ranges over all sequences of even integers  such that $ n_i\geq 2r$ for $1\leq i\leq a+b$, the set of points
$(z_1(\lambda),z_2(\lambda), \ldots, z_{a+b}(\lambda))$ defined by \eqref{zrlam} is Zariski dense in $\mathbb C^{a+b}$.
\end{Cor}
\begin{proof}Thanks to Lemma~\ref{dominant}, we have the result as required.
\end{proof}
For any positive  integer $b$,  let $\mathcal Z:=\mathbb Z[\hat u_1,\cdots, \hat u_b]$ be the ring of polynomials in variables $\hat u_1,\cdots, \hat u_b$.

\begin{Theorem}\label{mainf} Suppose  $r$ is a positive integer. Let $\OBC^{\hat f}_{\mathcal Z}$ be the cyclotomic oriented Brauer-Clifford supercategory over $\mathcal Z$ with respect to  $\hat f(t)=t^{2a+\epsilon}\prod_{i=1}^b(t^2-\hat u_i)\in\mathcal Z[t]$.    Let   $S$   be the set of  all elements $$
p_d(\Delta_1, \Delta_3, \ldots, \Delta_{2(a+b)-1})\undot{d}\circ x_r^{\beta_r(d)}\circ\cdots\circ x_1^{\beta_1(d)},$$
where \begin{enumerate} \item $d$ ranges over all equivalence classes of  normally ordered dotted oriented Brauer-Clifford diagrams without bubbles of type  $\up^r\to\up^r$,\item   $p_d(t_1,t_3,\ldots,t_{2(a+b)-1})$ ranges over all  polynomials in $ \mathcal Z[t_1, \ldots, t_{2(a+b)-1}]$,\item
 $(\beta_r(d), \beta_{r-1}(d), \ldots, \beta_1(d))\in \underline \ell^r$, where $\underline \ell=\{0, 1, \ldots, \ell-1\}$.\end{enumerate}
\noindent For a finite subset $A\subset  S$, $$\sum p_d(\Delta_1, \Delta_3, \ldots, \Delta_{2(a+b)-1}) d=0$$ (where the summation is over all elements in $A$) if and only if  $p_d(t_1,  \ldots, t_{2(a+b)-1} )=0$ for all  $p_d(\Delta_1, \Delta_3, \ldots, \Delta_{2(a+b)-1}) d  \in A$.
 \end{Theorem}
\begin{proof}Suppose $ \sum p_d(\Delta_1, \Delta_3, \ldots, \Delta_{2(a+b)-1}) d=0$ in $\OBC^{\hat f}_{\mathcal Z}$.  There is a ring homomorphism from $\mathcal Z$ to $\mathbb C$, sending $\hat u_i$ to $u_i$ for all admissible $i$, where
$u_i$'s are given in Theorem~\ref{zerop}. We have $$ \sum p_d(\Delta_1, \Delta_3, \ldots, \Delta_{2(a+b)-1}) d\otimes_{\mathcal Z } 1=0 \text{ in $\OBC^{f}_{\mathbb C}$.}$$ Applying  the    superfunctor
$\Psi_{M^{\mathfrak p}(\l)}^f$ on the previous equation (see Theorem~\ref{zerop}) and using   Proposition~\ref{basis theorem for level l Hecke-Clifford}, we have $$ p_d(\delta_1, \delta_3, \ldots, \delta_{2(a+b)-1}) =0,$$ where
$p_d(\delta_1, \delta_3, \ldots, \delta_{2(a+b)-1})$ is obtained from
$p_d(\Delta_1, \Delta_3, \ldots, \Delta_{2(a+b)-1}) $ by replacing
$u_i$ (resp., $\Delta_{2j-1}$) with $l_{a+i}(l_{a+i}+1)$ (resp., $\delta_{2j-1}$). Thanks to Corollary~\ref{zeropoly},   $\tilde p_{d}(t_1,t_3,\ldots, t_{2(a+b)-1})=0$, where $\tilde p_{d}(t_1,t_3,\ldots, t_{2(a+b)-1})$ is obtained from $p_{d}(t_1,t_3,\ldots, t_{2(a+b)-1})$ by specializing $\hat u_i$ at $l_{a+i}(l_{a+i}+1)$, $1\le i\le b$. By Theorem~\ref{zerop} and \eqref{de of lam}, there are infinite choices of $l_{a+1}, l_{a+2}, \ldots, l_{a+b}$ such that   $\tilde p_{d}(t_1,t_3,\ldots, t_{2(a+b)-1})=0$. Further, the choices of $l_{a+i}$ and $l_{a+j}$ are independent whenever $i\neq j$. Thanks to the fundamental theorem of  algebra, we have $p_{d}(t_1,t_3,\ldots, t_{2(a+b)-1})=0$ in $\mathcal Z[t_1, t_3, \cdots, t_{2(a+b)-1}]$.
\end{proof}

\begin{Cor}\label{anyring}
For any nonnegative  $r$ and commutating ring $\kappa$  containing $2^{-1}$, $\End_{\OBC_\kappa^f}(\uparrow^r)$ has basis given by  all equivalence classes of normally ordered dotted oriented Brauer-Clifford diagrams with bubbles of type $\up^r\to\up^r$ with fewer than $\ell$ $\fulldot$'s on each strand.
\end{Cor}
\begin{proof}In \cite{CK},  Comes and Kujawa have proved that $\End_{\OBC_\kappa^f }(\uparrow^r)$ is spanned  by all equivalence classes of normally ordered dotted oriented Brauer-Clifford diagrams with bubbles of type $\uparrow^r\to\uparrow^r$ with fewer than $\ell$ $\fulldot$'s on each strand. When $\kappa=\mathcal Z$, the linear independent of such elements  immediately follows from Theorem~\ref{mainf}. This proves the result when $\kappa=\mathcal Z$.  In general, we consider
$\kappa$ as the $\mathcal Z$-module on which  $\hat u_i$'s  act as scalars $u_i$'s for all $1\le i\le b$.
By arguments similar to those for $\AOBC_\kappa$ in \cite{CK},  one can define the  $\kappa$-supercategory $\OBC_{\mathcal Z}^{\hat f} \otimes_{\mathcal Z}\kappa $
with the objects as $\OBC_{\mathcal Z}^{\hat f} $, and the morphisms are
\begin{equation}\label{iso321} \Hom_{ \OBC_{\mathcal Z}^{\hat f} \otimes_{\mathcal Z}\kappa}(\ob a, \ob b)= \Hom_{ \OBC_{\mathcal Z}^{\hat f}}(\ob a, \ob b)\otimes _{\mathcal Z}\kappa.\end{equation} The obvious mutually inverse superfunctors  provide an isomorphism of  supercategories between $\OBC_{\mathcal Z}^{\hat f} \otimes_{\mathcal Z}\kappa $ and $\OBC_\kappa^f$.
The result follows from the   base change, immediately.
\end{proof}

 The following result can be proven by arguments similar to those for $\AOBC_\kappa$ in \cite[\S6]{CK}. The difference is that we have to consider $\OBC_{\mathcal Z}^f$ whereas they can
consider $\AOBC_{\mathbb Z}$.

\begin{Theorem}\label{Cyclotomic basis conjecture1} Conjecture \ref{Cyclotomic basis conjecture} is true over an arbitrary commutative ring  $\kappa$  containing $2^{-1}$. \end{Theorem}

\begin{proof} Suppose that  $\ob a $ (resp. $\ob b$) consists of $r_1$ (resp., $r_1'$) $\uparrow$'s and $r_2$ (resp., $r_2'$ $\downarrow$'s). If $r_1+r_2'\neq r_1'+r_2$, then there is no oriented Brauer diagram of type $\ob a \rightarrow \ob b$, forcing
$\Hom_{\OBC^f} (\ob a, \ob b)=0$.  When $ r_1+r_2'= r_1'+r_2:=r$, and $\kappa$ is a field,  there is a $\kappa$-linear isomorphism
\begin{equation}\label{isospace}
\Hom_{\OBC^f_\kappa} (\ob a, \ob b)_{\leq k}\rightarrow \Hom_{\OBC_\kappa^f} (\downarrow^{r_2}\uparrow^{r_1}, \uparrow^{r_1'}\downarrow^{r_2'})_{\leq k}\rightarrow
\End_{\OBC_\kappa^f} (\uparrow^r)_{\leq k}\end{equation}
defined  in the same way as the top horizontal maps in \cite[(4.2)]{CK}, where $\Hom_{\OBC_\kappa^f} (\ob a, \ob b)_{\leq k}$ is the $\kappa$-span of all dotted oriented Brauer-Clifford diagrams with bubbles of type $\ob a\to\ob b$ having at most $k$ $\fulldot$'s, and $0\leq k\leq \ell-1$. By \eqref{isospace},  $$\dim_\kappa \Hom_{\OBC_\kappa^f} (\ob a, \ob b)=\sum_{k=0}^{\ell-1}\dim_\kappa \Hom_{\OBC_\kappa^f} (\ob a, \ob b)_{\leq k}, $$ forcing
$\dim_\kappa \Hom_{\OBC_\kappa^f} (\ob a, \ob b)=\dim_\kappa \End_{\OBC_\kappa^f} (\uparrow^r)$.   Comes-Kujawa proved that $\Hom_{\OBC_\kappa^f} (\ob a, \ob b)$ is spanned by the set of all equivalence classes of normally ordered dotted oriented Brauer-Clifford diagrams with bubbles of type $\ob a\to\ob b$ with fewer than $\ell$ $\fulldot$'s on each strand whenever $\kappa$ is a commutative ring containing $2^{-1}$,
 Corollary~\ref{anyring} immediately implies Theorem~\ref{Cyclotomic basis conjecture1} over the field $\kappa$. Since  the $\mathcal Z$-linear independent of the set of all equivalence classes of  normally ordered dotted Brauer-Clifford diagrams of type $\ob a\rightarrow \ob b $ follows from the corresponding result over the fraction field of $\mathcal Z$,  we have Theorem~\ref{Cyclotomic basis conjecture1} over $\mathcal Z$. By \eqref{iso321}, we have  Theorem~\ref{Cyclotomic basis conjecture1} over an arbitrary commutative ring $\kappa$ containing $2^{-1}$.\end{proof}

\section{Cyclotomic walled Brauer-Clifford superalgebras}
The aim of this section is to establish  connections between two cyclotomic walled Brauer-Clifford superalgebras defined in \cite{CK, GRSS}.
The level two cases has been dealt with in \cite{GRSS} under the assumption that
$f(t)=t^2-u$ with $u\neq 0$ over the complex field.

 We start by recalling the notion of cyclotomic walled Brauer-Clifford superalgebras in \cite{GRSS}.
Let $\Sigma_r$ be the {\it symmetric group} in $r$ letters.
 Then $\Sigma_r$ is generated by $s_1, \ldots, s_{r-1}$,  subject to the  relations (for all admissible $i$ and $j$):
\begin{equation}\label{symm} s_i^2=1, \ \ s_is_{i+1}s_i=s_{i+1}s_is_{i+1}, \ \ \text{ $s_is_j=s_js_i$, if  $|i-j|>1$.}\end{equation}
The {\it  Hecke-Clifford algebra} $HC_{r}$ is  the  associative $\kappa$-superalgebra generated by even elements $s_1, \ldots, s_{r-1}$ and odd elements $c_1, \ldots, c_r$   subject to \eqref{symm} together with  the   following defining relations (for all admissible  $i, j$):
\begin{equation}\label{sera} c_i^2=-1, \ \ c_ic_j=-c_jc_i, \ \ \text{$w^{-1} c_i w=c_{(i)w}, \forall w\in \Sigma_r$}. \end{equation}
The {\it affine Hecke-Clifford} algebra
$HC_{r}^{\rm aff}$ is the associative $\kappa$-superalgebra generated by even elements  $s_1, \ldots, s_{r-1},\ob x_1$ and odd elements $c_1, \ldots, c_r$   subject to
 \eqref{symm}--\eqref{sera}, together with the following defining relations (for all admissible $i$ and $j$):
 \begin{equation} \label{asera}\ob x_1\ob x_2=\ob x_2\ob x_1, \ \ \ob x_1c_i=(-1)^{\delta_{i, 1}} c_i\ob x_1, \ \ s_j\ob x_1=\ob x_1s_j, \text{if $j\neq 1$, }\end{equation}
 where $\ob x_2=s_1\ob x_1s_1-(1-c_1c_2)s_1$.

Let $\overline{HC}_r$  be the $\kappa$-superalgebra  generated by the even elements $\bar s_1, \ldots, \bar s_{r-1}$ and odd  elements $\bar c_1, \ldots, \bar c_r$ subject to the relations (for all admissible $i$ and $j$):  \begin{equation}\label{aseradual}\begin{aligned} &  \bar s_i^2=1, \ \ \bar s_i\bar s_{i+1}\bar s_i=\bar s_{i+1}\bar s_i\bar s_{i+1}, \ \ \text{and  $\bar s_i\bar s_j=\bar s_j\bar s_i$, if  $|i-j|>1$, }\\
 & \bar c_i^2=1, \ \ \bar c_i\bar c_j=-\bar c_j\bar c_i, \ \ \text{and $w^{-1} \bar c_i w=\bar c_{(i)w}, \forall w\in \Sigma_r$. }\\ \end{aligned}  \end{equation}
Let $\overline{HC}_r^{\rm aff}$ be the $\kappa$-superalgebra generated  by even elements  $\bar s_1, \ldots, \bar s_{r-1}, \bar {\ob x}_1$ and odd elements $\bar c_1, \ldots, \bar c_r$   subject to \eqref{aseradual} together with the following defining relations (for all admissible $i$ and $j$):
 \begin{equation}\label{aseradual2}
 \bar {\ob x}_1\bar {\ob x}_2=\bar {\ob x}_2\bar {\ob x}_1, \ \  \bar {\ob x}_1\bar c_i=(-1)^{\delta_{i, 1}} \bar c_i\bar {\ob x}_1, \ \ \bar s_j\bar {\ob x}_1=\bar {\ob x}_1\bar s_j, \text{if $j\neq 1$, }
 \end{equation}
 where $\bar {\ob x}_2=\bar s_1\bar {\ob x}_1\bar s_1-(1+\bar c_1\bar c_2)\bar s_1$.

 In \cite[Theorem~5.1]{JK}, Jung and Kang defined   the walled  Brauer-Clifford superalgebra $ BC_{r,t}$. It is  the associative $\kappa$-superalgebra  generated by even generators $e_1$,  $s_1, \ldots, s_{r-1}$, $\bar s_{1}, \ldots, \bar s_{t-1}$, and odd generators  $c_1, \ldots, c_r, \bar c_1, \ldots,  \bar c_{t}$ subject to  \eqref{symm},\eqref{sera} and \eqref{aseradual} together with the following defining relations for all admissible $i, j$:
 \begin{multicols}{2}
\begin{enumerate}
\item [(1)] $e_1 c_1=e_1\bar c_1 $, $c_1 e_1=\bar c_1 e_1 $,
\item  [(2)]  $\bar s_j c_i=c_i \bar s_j$, $ s_i \bar c_j=\bar c_j  s_i$,
\item [(3)] $c_i \bar c_j=-\bar c_j c_i$, $s_i\bar s_j=\bar s_j s_i$,
\item [(4)]  $e_1^2=0$,
\item [(5)] $e_1 s_1 e_1=e_1=e_1\bar s_1 e_1$,
\item [(6)] $s_i e_1=e_1 s_i$,   $\bar s_i e_1=e_1 \bar s_i$, if $i\neq 1$,
\item [(7)]$e_1s_1\bar s_1 e_1 s_1=e_1s_1\bar s_1 e_1 \bar s_1$,
\item [(8)] $s_1 e_1s_1\bar s_1 e_1 =\bar s_1 e_1s_1\bar s_1 e_1$,
\item [(9)] $c_i e_1=e_1 c_i$ and  $\bar c_i e_1=e_1\bar c_i$, if $i\neq 1$,
\item [(10)] $e_1c_1e_1=0=e_1\bar c_1 e_1$.
\end{enumerate}
\end{multicols}

\begin{Defn}\label{awbsa}\cite[Definition~3.1]{GRSS}
 The {\it affine  walled Brauer-Clifford superalgebra}
$ BC_{r,t}^{\text{\rm aff}}$ is the associative $\kappa$-superalgebra generated by odd elements $c_1, \ldots, c_r$, $\bar c_1, \ldots, \bar c_t$ and  even elements  $e_1, \ob x_1, \bar {\ob x}_1$, $s_1, \ldots, s_{r-1}$, $\bar s_1, \ldots, \bar s_{t-1}$,  and two families of even central elements  $\omega_{2k+1}, \bar \omega_{k}$, $ k\in\Z^{\ge1}$ subject to \eqref{symm}--\eqref{aseradual2} and the above relations (1)--(10) together with the following defining relations for all admissible $i$:
\begin{multicols}{2}
\begin{enumerate}
\item [(1)]$e_1(\ob x_1+\bar{\ob x}_1)=(x_1+\bar {\ob x}_1)e_1=0$,
 \item [(2)] $e_1s_1\ob x_1s_1=s_1\ob x_1s_1e_1$,
\item [(3)] $\ob x_1(e_1+\bar{\ob x}_1-\bar e_1 )=(e_1-\bar e_1+\bar {\ob x}_1)x_1$,
\item [(4)] $e_1\bar s_1\bar {\ob x}_1\bar s_1=\bar s_1\bar {\ob x}_1\bar s_1 e_1$,
\item [(5)] $e_1{\ob x}_1^{2k+1} e_1=\omega_{2k+1} e_1$,  $\forall k\in \mathbb N$,
\item [(6)] $e_1\ob x_1^{2k}e_1=0$,  $\forall k\in \mathbb N$,
 \item  [(7)] $e_1\bar {\ob x}_1^{k}e_1=\bar \omega_{k}e_1$, $\forall k\in \mathbb Z^{>0}$,
 \item [(8)] $\ob x_1 \bar c_{i}=\bar c_{i} \ob x_1$,
 \item [(9)] $\bar {\ob x}_1 c_i= c_i \bar {\ob x}_1$,
 \item [(10)] $ \ob x_1 \bar s_i=\bar s_i \ob x_1$,
 \item [(11)] $\bar {\ob x}_1 s_i=s_i \bar {\ob x}_1$.
 \end{enumerate}
\end{multicols}
 \end{Defn}
By \cite[Lemma~3.4, Corollary~3.5]{GRSS},  we have to assume that   $\bar \omega_{k}$'s satisfy some technical conditions which are defined via $\omega_{2i+1}$ for all $i\in \mathbb N$ and moreover,
 $\bar \omega_{2k}=0$ for $k\in\mathbb Z_{\geq 0}$.
Otherwise, $e_1=0$ and  $ BC_{r,t}^{\text{\rm aff}}$ is isomorphic to the outer  tensor product of $HC^{\rm aff}_r$ and $\overline{HC}^{\rm aff}_t$ whenever $\kappa$ is a field.

Let $\AOBC_\kappa(\delta)$ and $\OBC_\kappa^f(\delta)$ be the category obtained from
$\AOBC_\kappa $ and $\OBC_\kappa^f$ by specializing $\Delta_k$ at $\delta_k$, where
$\delta=(\delta_k)_{k\in\mathbb Z_{\geq0}}$.
It is proven   in \cite{CK, GRSS} that there is
 a $\kappa$-superalgebra homomorphism
 \begin{equation} \label{isomofoabcaff} \varphi: {BC}_{r,t}^{\rm aff}\longrightarrow \End_{\AOBC_\kappa(\delta)}(\downarrow^t\uparrow^r).   \end{equation}
 Since we are going to use the same notation in  \cite{CK},  we use their  homomorphism $ \varphi$ in \eqref{isomofoabcaff}. By \cite[A.4]{CK}, $\varphi$ is defined by
$$\begin{aligned}
    c_i & \mapsto\sqrt{-1}~
        \begin{tikzpicture}[baseline = 7.5pt, scale=0.5, color=\clr]
            \draw[<-,thick] (0,0) to[out=up, in=down] (0,1.5);
        \end{tikzpicture}^{~s ~}
        \begin{tikzpicture}[baseline = 7.5pt, scale=0.5, color=\clr]
            \draw[->,thick] (0,0) to[out=up, in=down] (0,1.5);
        \end{tikzpicture}^{~r-i}
        \begin{tikzpicture}[baseline = 7.5pt, scale=0.5, color=\clr]
            \draw[->,thick] (0,0) to[out=up, in=down] (0,1.5);
            \draw (0,0.7) \wdot;
        \end{tikzpicture} \;\;
        \begin{tikzpicture}[baseline = 7.5pt, scale=0.5, color=\clr]
            \draw[->,thick] (0,0) to[out=up, in=down] (0,1.5);
        \end{tikzpicture}^{~i-1} &
    \bar{c}_j &\mapsto\sqrt{-1}~
        \begin{tikzpicture}[baseline = 7.5pt, scale=0.5, color=\clr]
            \draw[<-,thick] (0,0) to[out=up, in=down] (0,1.5);
        \end{tikzpicture}^{~s-j}
        \begin{tikzpicture}[baseline = 7.5pt, scale=0.5, color=\clr]
            \draw[<-,thick] (0,0) to[out=up, in=down] (0,1.5);
            \draw (0,0.7) \wdot;
        \end{tikzpicture} \;\;
        \begin{tikzpicture}[baseline = 7.5pt, scale=0.5, color=\clr]
            \draw[<-,thick] (0,0) to[out=up, in=down] (0,1.5);
        \end{tikzpicture}^{~j-1~}
        \begin{tikzpicture}[baseline = 7.5pt, scale=0.5, color=\clr]
            \draw[->,thick] (0,0) to[out=up, in=down] (0,1.5);
        \end{tikzpicture}^{~r} \\
    e_1 & \mapsto
        \begin{tikzpicture}[baseline = 11pt, scale=0.5, color=\clr]
            \draw[<-,thick] (0,0) to[out=up, in=down] (0,2);
        \end{tikzpicture}^{~s-1}
        \begin{tikzpicture}[baseline = 11pt, scale=0.5, color=\clr]
            \draw[<-,thick] (0,0) to[out=up,in=left] (2,0.75) to[out=right,in=up] (4,0);
            \draw[->,thick] (0,2) to[out=down,in=left] (2,1.25) to[out=right,in=down] (4,2);
            \draw[->,thick] (2,0) to[out=up, in=down] (2,2);
            \draw[color=black] (2.75,1.9) node{$~^{r-1}$};
        \end{tikzpicture}
        \\
   \ob x_1 & \mapsto -
        \begin{tikzpicture}[baseline = 7.5pt, scale=0.5, color=\clr]
            \draw[<-,thick] (0,0) to[out=up, in=down] (0,1.5);
        \end{tikzpicture}^{~s~}
        \begin{tikzpicture}[baseline = 7.5pt, scale=0.5, color=\clr]
            \draw[->,thick] (0,0) to[out=up, in=down] (0,1.5);
        \end{tikzpicture}^{~r-1}
        \begin{tikzpicture}[baseline = 7.5pt, scale=0.5, color=\clr]
            \draw[->,thick] (0,0) to[out=up, in=down] (0,1.5);
            \draw (0,0.7) \bdot;
        \end{tikzpicture}
        &
    \bar{\ob x}_1 &\mapsto
        \begin{tikzpicture}[baseline = 11pt, scale=0.5, color=\clr]
            \draw[<-,thick] (0,0) to[out=up, in=down] (0,2);
        \end{tikzpicture}^{~s-1}
        \begin{tikzpicture}[baseline = 11pt, scale=0.5, color=\clr]
            \draw[->,thick] (0,2) to[out=down, in=up] (1.5,1) to[out=down,in=up] (0,0);
            \draw[->,thick] (0.75,0) to (0.75,2);
            \draw[color=black] (1,1.9) node{$~^{r}$};
            \draw (1.5,1) \bdot;
        \end{tikzpicture}
        \\
    s_i & \mapsto
        \begin{tikzpicture}[baseline = 7.5pt, scale=0.5, color=\clr]
            \draw[<-,thick] (0,0) to[out=up, in=down] (0,1.5);
        \end{tikzpicture}^{~s~}
        \begin{tikzpicture}[baseline = 7.5pt, scale=0.5, color=\clr]
            \draw[->,thick] (0,0) to[out=up, in=down] (0,1.5);
        \end{tikzpicture}^{~r-i-1}
        \begin{tikzpicture}[baseline = 7.5pt, scale=0.5, color=\clr]
            \draw[->,thick] (0,0) to[out=up, in=down] (1,1.5);
            \draw[->,thick] (1,0) to[out=up, in=down] (0,1.5);
        \end{tikzpicture} \;\;
        \begin{tikzpicture}[baseline = 7.5pt, scale=0.5, color=\clr]
            \draw[->,thick] (0,0) to[out=up, in=down] (0,1.5);
        \end{tikzpicture}^{~i-1} &
    \bar{s}_j &\mapsto
        \begin{tikzpicture}[baseline = 7.5pt, scale=0.5, color=\clr]
            \draw[<-,thick] (0,0) to[out=up, in=down] (0,1.5);
        \end{tikzpicture}^{~s-j-1}
        \begin{tikzpicture}[baseline = 7.5pt, scale=0.5, color=\clr]
            \draw[<-,thick] (0,0) to[out=up, in=down] (1,1.5);
            \draw[<-,thick] (1,0) to[out=up, in=down] (0,1.5);
        \end{tikzpicture} \;\;
        \begin{tikzpicture}[baseline = 7.5pt, scale=0.5, color=\clr]
            \draw[<-,thick] (0,0) to[out=up, in=down] (0,1.5);
        \end{tikzpicture}^{~j-1~}
        \begin{tikzpicture}[baseline = 7.5pt, scale=0.5, color=\clr]
            \draw[->,thick] (0,0) to[out=up, in=down] (0,1.5);
        \end{tikzpicture}^{~r~} \\
 \omega_{2k+1} & \mapsto -\delta_{2k+1} &
   \bar{\omega}_k &\mapsto \begin{cases}    \delta'_k, &\text{if $k$ is odd;}\\
                                        0, &\text{if $k$ is even.}
                       \end{cases}\\
\end{aligned}
$$
where $\delta_k'$'s $\in \kappa$ are   determined by  \begin{equation}\label{del'}\delta_k'-\delta_k= -\sum_{0< i< k/2}\delta_{2i-1}\delta_{k-2i}'.\end{equation}
Via the previous homomorphism, we see that  \eqref{del'} are the same as those relations in \cite[Corollary~3.5]{GRSS}.
Let $ {BC}_{r,t}^{\rm aff}(\delta)$ be  the affine walled Brauer-Clifford superalgebra obtained from $ {BC}_{r,t}^{\rm aff}$ by specializing    central generators $ \omega_{2k+1}$ and $\bar \omega_{2k+1} (k\in\mathbb Z_{>0})$ at $-\delta_{2k+1}$  and $\delta'_k $ in $\kappa$. Then $\varphi$ factors through ${BC}_{r,t}^{\rm aff}(\delta)$. The basis theorem of $ \End_{\AOBC_\kappa(\delta) }(\downarrow^t\uparrow^r)$  in \cite{CK} yields the following
 $\kappa$-superalgebra isomorphism induced from $\varphi$:
 \begin{equation} \label{isomofoabc} \bar\varphi: {BC}_{r,t}^{\rm aff}(\delta)\longrightarrow {\mathscr {BC}}_{r, t}^{\rm aff},   \end{equation}
where $ {\mathscr {BC}}_{r, t}^{\rm aff}:=\End_{\AOBC_\kappa(\delta) }(\downarrow^t\uparrow^r)$.
In particular, $\bar\varphi$ sends each  \emph{regular monomial} of  ${BC}_{r,t}^{\rm aff}(\delta)$ in \cite[Definition~3.15]{GRSS} to a unique  equivalence class of normally ordered dotted oriented  Brauer-Clifford diagram (without bubbles).
The tiny difference  between the isomorphisms established in \cite{CK, GRSS} is that we  number the leftmost strand as the first one in \cite{GRSS} while Comes and Kujawa number the rightmost as the first strand (see, e.g, \eqref{xlabel}). Moreover, we use right tensor ideal in \cite{GRSS} while they use left tensor ideal for the definition of $\OBC_\kappa^f$.

\begin{Defn}
Define two linear  $\kappa$-linear homomorphisms  $\sigma_\uparrow: \End_{\AOBC_\kappa}(\uparrow)\rightarrow\End_{\AOBC_\kappa}(\downarrow)$ and
$\sigma_\downarrow: \End_{\AOBC_\kappa}(\downarrow)\rightarrow\End_{\AOBC_\kappa}(\uparrow)$ such that, for any
$h_1\in \End_{\AOBC_\kappa}(\uparrow)$ and $ h_2\in \End_{\AOBC_\kappa}(\downarrow)$,
$$\begin{aligned} & \sigma_\uparrow(h_1):= (\downarrow\lcap)\circ (\dswap h_1)\circ ( \downarrow \rcup)= \begin{tikzpicture}[baseline = 15pt, scale=0.5, color=\clr]
        \draw[<-,thick] (0,0) to (0,0.5) to(1,1);
        \draw[-,thick] (2,1) to[out=up, in=right] (1.53,1.5) to[out=left, in=right] (1.47,1.5);
        \draw[-,thick] (1.49,1.5) to[out=left,in=up] (1,1);
          \draw[-,thick]  (2,1)to(2,0.7);
             \draw[-,thick] (2,0.7) to[out=down, in=right] (1.53,0.2) to[out=left, in=right] (1.47,0.2);
        \draw[-,thick] (1.49,0.2) to[out=left,in=down] (1,0.7);
        \draw[-,thick]   (1,0.7)to(0,1)to(0,2);
       \draw (2,0.8) node[circle,draw,thick,fill=white,inner sep=0pt, minimum width=10pt]{\tiny$h_1$};
    \end{tikzpicture},\\
    & \sigma_\downarrow(h_2):= (\uparrow\lcap)\circ (\swap h_2)\circ ( \uparrow \rcup)=  \begin{tikzpicture}[baseline = 15pt, scale=0.5, color=\clr]
        \draw[-,thick] (0,0) to (0,0.5) to(1,1);
        \draw[-,thick] (2,1) to[out=up, in=right] (1.53,1.5) to[out=left, in=right] (1.47,1.5);
        \draw[-,thick] (1.49,1.5) to[out=left,in=up] (1,1);
          \draw[-,thick]  (2,1)to(2,0.7);
             \draw[-,thick] (2,0.7) to[out=down, in=right] (1.53,0.2) to[out=left, in=right] (1.47,0.2);
        \draw[-,thick] (1.49,0.2) to[out=left,in=down] (1,0.7);
        \draw[->,thick]   (1,0.7)to(0,1)to(0,2);
        \draw (2,0.8) node[circle,draw,thick,fill=white,inner sep=0pt, minimum width=10pt]{\tiny$h_2$};
    \end{tikzpicture}.\\ \end{aligned}$$

\end{Defn}
\begin{Lemma}\label{sigmaupd}
$\sigma_\uparrow$  and $\sigma_\downarrow$ are mutually inverse to each other.
\end{Lemma}
\begin{proof}The result  immediately  follows from the first relation in \eqref{OB relations 1 (symmetric group)}.
\end{proof}

Similarly, we  define $$\begin{aligned} & \sigma_\uparrow(\delta): \End_{\AOBC_\kappa(\delta)}(\uparrow)\rightarrow\End_{\AOBC_\kappa(\delta)}(\downarrow), \\   & \sigma_\downarrow(\delta): \End_{\AOBC_\kappa(\delta)}(\downarrow)\rightarrow\End_{\AOBC_\kappa(\delta)}(\uparrow)\\
\end{aligned}$$
in an obvious way. Recall $f(t)$ is given in \eqref{funcf}.
\begin{Lemma}\label{gtp} There is a
 $g(t)\in \kappa[t]$ such that the equation
 \begin{equation}\label{ind321} (-1)^\ell\bar\varphi(e_1f(\ob x_1))=\bar\varphi(e_1g(\bar{\ob x}_1))\end{equation}
holds  in  $\AOBC_\kappa(\delta)$.
\end{Lemma}
\begin{proof}
Write  $f(t)=a_\ell t^\ell+a_{\ell-1}t^{\ell-1}+\ldots+a_1 t+a_0$, where $a_i\in \kappa$ for all admissible $i$.
Define $y_k:= \sigma_\uparrow(\xdotk )$ in $\AOBC_\kappa$ for $1\leq k\leq \ell$.
By \cite[(3.26)]{CK},
 $y_k=\begin{tikzpicture}[baseline = 15pt, scale=0.5, color=\clr]
        \draw[<-,thick] (0,0.5)  to(0,1.8);
  \draw(0,1) \bdot; \draw(0.4,1)node{$k$};     \end{tikzpicture} +\sum_{i=0}^{k-1}\Delta_{k-i-1}
\begin{tikzpicture}[baseline = 15pt, scale=0.5, color=\clr]
        \draw[<-,thick] (0,0.4)  to(0,1.8);
  \draw(0,1.3) \bdot; \draw(0.4,1.3)node{$ i$};     \end{tikzpicture}$.     So,
\begin{equation}\label{ykde}
y_k= \begin{tikzpicture}[baseline = 15pt, scale=0.5, color=\clr]
        \draw[<-,thick] (0,0.5)  to(0,1.8);
  \draw(0,1) \bdot; \draw(0.4,1)node{$k$};     \end{tikzpicture} +\sum_{i=0}^{k-1}\delta_{k-i-1}
\begin{tikzpicture}[baseline = 15pt, scale=0.5, color=\clr]
        \draw[<-,thick] (0,0.4)  to(0,1.8);
  \draw(0,1.3) \bdot; \draw(0.4,1.3)node{$i$};     \end{tikzpicture}
\end{equation} in ${\AOBC_\kappa(\delta)}$.
As $\begin{tikzpicture}[baseline = 15pt, scale=0.5, color=\clr]
        \draw[<-,thick] (0,0.5)  to(0,1.8);
  \draw(0,1) \bdot; \draw(0.4,1)node{$k$};     \end{tikzpicture}=(\xdotr)^k $, $y_k$ can be considered as  a polynomial of $\xdotr$ with degree $k$.
We define \begin{equation}\label{ggg} \mathbf g:=\sigma_\uparrow(\delta)(f(\xdot) )  =
    \begin{tikzpicture}[baseline = 15pt, scale=0.5, color=\clr]
        \draw[<-,thick] (0,0) to (0,0.5) to(1,1);
        \draw[-,thick] (2,1) to[out=up, in=right] (1.53,1.5) to[out=left, in=right] (1.47,1.5);
        \draw[-,thick] (1.49,1.5) to[out=left,in=up] (1,1);
          \draw[-,thick]  (2,1)to(2,0.7);
             \draw[-,thick] (2,0.7) to[out=down, in=right] (1.53,0.2) to[out=left, in=right] (1.47,0.2);
        \draw[-,thick] (1.49,0.2) to[out=left,in=down] (1,0.7);
        \draw[-,thick]   (1,0.7)to(0,1)to(0,2);\draw (2,0.8) node[circle,draw,thick,fill=white,inner sep=0pt, minimum width=10pt]{\tiny$j$};
    \end{tikzpicture} \end{equation}in ${\AOBC_\kappa(\delta)}$,
  where $j=f(\xdot)$.
     So $ \mathbf g= a_\ell y_\ell+a_{\ell-1}y_{\ell-1}+\ldots+a_1 y_1+a_0$, which can be considered as   a polynomial  of $\xdotr$ with degree $\ell$.
Let   $g(t)\in\kappa[t]$ be obtained from $\mathbf g$ by replacing $\xdotr$ by $t$. Then
$g(\xdotr)=\mathbf g$.
 Thanks to Lemma~\ref{sigmaupd}, $\sigma_\downarrow(\delta)(g(\xdotr) )=\sigma_\downarrow(\delta)\circ \sigma_\uparrow(\delta) (f(\xdot) )=f(\xdot)$. Let   $h:=g(\xdotr)$. Then
\begin{equation}\label{fgf}
f(\xdot)= \begin{tikzpicture}[baseline = 15pt, scale=0.5, color=\clr]
        \draw[-,thick] (0,0) to (0,0.5) to(1,1);
        \draw[-,thick] (2,1) to[out=up, in=right] (1.53,1.5) to[out=left, in=right] (1.47,1.5);
        \draw[-,thick] (1.49,1.5) to[out=left,in=up] (1,1);
          \draw[-,thick]  (2,1)to(2,0.7);
             \draw[-,thick] (2,0.7) to[out=down, in=right] (1.53,0.2) to[out=left, in=right] (1.47,0.2);
        \draw[-,thick] (1.49,0.2) to[out=left,in=down] (1,0.7);
        \draw[->,thick]   (1,0.7)to(0,1)to(0,2);\draw (2,0.8) node[circle,draw,thick,fill=white,inner sep=0pt, minimum width=10pt]{\tiny$h$};
    \end{tikzpicture},
\end{equation}
Via  \eqref{OB relations 1 (symmetric group)}--\eqref{OB relations 2 (zigzags and invertibility)} and the definition of $\varphi$,
we have
\begin{equation}\label{bark}
\bar\varphi(\bar {\ob x}_1)^k=  \begin{tikzpicture}[baseline = 11pt, scale=0.5, color=\clr]
            \draw[<-,thick] (0,0) to[out=up, in=down] (0,2);
        \end{tikzpicture}^{~s-1}
        \begin{tikzpicture}[baseline = 11pt, scale=0.5, color=\clr]
            \draw[->,thick] (0,2) to[out=down, in=up] (1.5,1) to[out=down,in=up] (0,0);
            \draw[->,thick] (0.75,0) to (0.75,2);
            \draw[color=black] (1,1.9) node{$~^{r}$};
            \draw (1.5,1) \bdot;\draw(1.9,1)node{$k$};
        \end{tikzpicture} \quad \text { and hence }
        \bar\varphi(g(\bar{\ob x}_1))=\begin{tikzpicture}[baseline = 11pt, scale=0.5, color=\clr]
            \draw[<-,thick] (0,0) to[out=up, in=down] (0,2);
        \end{tikzpicture}^{~s-1}
        \begin{tikzpicture}[baseline = 11pt, scale=0.5, color=\clr]
            \draw[->,thick] (0,2) to[out=down, in=up] (1.5,1) to[out=down,in=up] (0,0);
            \draw[->,thick] (0.75,0) to (0.75,2);
            \draw[color=black] (1,1.9) node{$~^{r}$};\draw (1.5,1) node[circle,draw,thick,fill=white,inner sep=0pt, minimum width=10pt]{\tiny$h$};
        \end{tikzpicture}.
\end{equation}  Therefore, we have
\begin{equation}\label{varx1}
\bar\varphi ( f(\ob x_1))=(-1)^\ell f(\begin{tikzpicture}[baseline = 7.5pt, scale=0.5, color=\clr]
            \draw[<-,thick] (0,0) to[out=up, in=down] (0,1.5);
        \end{tikzpicture}^{~s~}
        \begin{tikzpicture}[baseline = 7.5pt, scale=0.5, color=\clr]
            \draw[->,thick] (0,0) to[out=up, in=down] (0,1.5);
        \end{tikzpicture}^{~r-1~}
        \begin{tikzpicture}[baseline = 7.5pt, scale=0.5, color=\clr]
            \draw[->,thick] (0,0) to[out=up, in=down] (0,1.5);
            \draw (0,0.7) \bdot;
        \end{tikzpicture}),
\end{equation}
In order to prove \eqref{ind321},  by \eqref{bark}--\eqref{varx1}, it is enough to prove that
\begin{equation}\label{fgfgfgf}
\begin{tikzpicture}[baseline = 7.5pt, scale=0.5, color=\clr]
            \draw[<-,thick] (0,0) to[out=up, in=down] (0,2);
        \end{tikzpicture}^{~s-1}
    \begin{tikzpicture}[baseline = 7.5pt, scale=0.5, color=\clr]
            \draw[<-,thick] (1,0) to[out=up, in=down] (1,2);
                    \draw[->,thick] (2,0) to[out=up, in=down] (2,2);\draw[color=black] (2.4,1.9) node{$~^{~r-1}$};
                       \draw[-,thick] (3.5,0) to[out=up, in=down] (3.5,2);
                       \draw[-,thick] (1,2) to[out=up,in=left] (2,2.75) to[out=right,in=up] (3.5,2);
            \draw (3.5,0.7) node[circle,draw,thick,fill=white,inner sep=0pt, minimum width=10pt]{\tiny$j$};
        \end{tikzpicture}=
        \begin{tikzpicture}[baseline = 11pt, scale=0.5, color=\clr]
            \draw[<-,thick] (0,0) to[out=up, in=down] (0,2);
        \end{tikzpicture}^{~s-1}
        \begin{tikzpicture}[baseline = 11pt, scale=0.5, color=\clr]
            \draw[->,thick] (0,2) to[out=down, in=up] (2,1) to[out=down,in=up] (0,0);
            \draw[->,thick] (0.75,0) to (0.75,2);
            \draw[color=black] (1,1.9) node{$~^{~r}$};
            \draw[-,thick] (1.5,0) to (1.5,2);
            \draw[-,thick] (0,2) to[out=up,in=left] (1,2.75) to[out=right,in=up] (1.5,2);
            \draw (2,1) node[circle,draw,thick,fill=white,inner sep=0pt, minimum width=10pt]{\tiny$h$};
        \end{tikzpicture}
\end{equation}
where $j=f(\xdot)$ and $h=g(\xdotr)$.
By \eqref{fgf}, we have
\begin{equation}\label{recupx}
\begin{tikzpicture}[baseline = 7.5pt, scale=0.5, color=\clr]
            \draw[<-,thick] (0,0) to[out=up, in=down] (0,2);
        \end{tikzpicture}^{~s-1}
    \begin{tikzpicture}[baseline = 7.5pt, scale=0.5, color=\clr]
            \draw[<-,thick] (1,0) to[out=up, in=down] (1,2);
                    \draw[->,thick] (2,0) to[out=up, in=down] (2,2);\draw[color=black] (2.4,1.9) node{$~^{~r-1}$};
                       \draw[-,thick] (3.5,0) to[out=up, in=down] (3.5,2);
                       \draw[-,thick] (1,2) to[out=up,in=left] (2,2.75) to[out=right,in=up] (3.5,2);
            \draw (3.5,0.7) node[circle,draw,thick,fill=white,inner sep=0pt, minimum width=10pt]{\tiny$j$};
        \end{tikzpicture}=\begin{tikzpicture}[baseline = 7.5pt, scale=0.5, color=\clr]
            \draw[<-,thick] (0,0) to[out=up, in=down] (0,2);
        \end{tikzpicture}^{~s-1}
    \begin{tikzpicture}[baseline = 7.5pt, scale=0.5, color=\clr]
            \draw[<-,thick] (1,0) to[out=up, in=down] (1,2);
                    \draw[->,thick] (2,0) to[out=up, in=down] (2,2);\draw[color=black] (2.4,1.9) node{$~^{~r-1}$};
                       %\draw[-,thick] (3.5,0) to[out=up, in=down] (3.5,2);
                       \draw[-,thick] (1,2) to[out=up,in=left] (2,2.75) to[out=right,in=up] (3.5,2);
           \draw[-,thick] (3.5,0) to (3.5,0.5) to(4.5,1);
        \draw[-,thick] (5.5,1) to[out=up, in=right] (5.03,1.5) to[out=left, in=right] (4.97,1.5);
        \draw[-,thick] (4.99,1.5) to[out=left,in=up] (4.5,1);
          \draw[-,thick]  (5.5,1)to(5.5,0.7);
             \draw[-,thick] (5.5,0.7) to[out=down, in=right] (5.03,0.2) to[out=left, in=right] (4.97,0.2);
        \draw[-,thick] (4.99,0.2) to[out=left,in=down] (4.5,0.7);
        \draw[-,thick]   (4.5,0.7)to(3.5,1)to(3.5,2);
            \draw (5.5,1) node[circle,draw,thick,fill=white,inner sep=0pt, minimum width=10pt]{\tiny$h$};
        \end{tikzpicture}
\end{equation}
 It follows from \cite[(3.4)]{CK} that
  \begin{equation}\label{scup}
  \begin{tikzpicture}[baseline = 5pt, scale=0.5, color=\clr]
        \draw[<-,thick] (0,0) to[out=up,in=left] (0.5,0.65) to[out=right,in=up] (1,0);
    \end{tikzpicture}
    ~=~
    \begin{tikzpicture}[baseline = 5pt, scale=0.5, color=\clr]
        \draw[<-,thick] (0,0) to[out=up,in=down] (1,1) to[out=up,in=right] (0.5,1.5) to[out=left,in=up] (0,1) to[out=down,in=up] (1,0);
    \end{tikzpicture}\end{equation}
Now \eqref{fgfgfgf} follows from \eqref{recupx}--\eqref{scup}.
\end{proof}

\begin{Lemma}\label{form} Suppose $\kappa$ is an algebraically closed field with characteristic not two.
Let $g(t)\in \kappa[t]$ be defined in the proof of Lemma~\ref{gtp}. Then there are two non-negative integers $a_1, b_1$ and some non-zero scalars $\bar u_i$ in $\kappa$, $1\le i\le b_1$  such that $a_1+2b_1=\ell$ and \begin{equation}\label{reform} g(\bar {\ob x}_1)=\bar x_1^{a_1}\prod_{j=1}^{b_1}(\bar {\ob x}_1^2-\bar u_j).\end{equation}
\end{Lemma}
\begin{proof}Since we are assuming that $\kappa$ is an algebraically closed field, we can write
$g(t)=t^{a_1}\prod_{j=1}^{c}(t-\bar v_j)$ such that $a_1+c=\ell$ and $\bar v_j\neq0$.
 \emph{This is the place where we have to assume that $\kappa$ is an algebraically closed field}.
 By Lemma~\ref{gtp} and the fact that $\bar \varphi$ is an isomorphism, we have
\begin{equation}\label{equatifg}
e_1f(\ob x_1))=(-1)^\ell e_1g(\bar{\ob x}_1)
\end{equation}
in ${BC}_{r,t}^{\rm aff}(\delta)$. Thanks to  \cite[Lemma~6.2]{GRSS}, we have  $\bar c_1 g(\bar{\ob x}_1)= (-1)^{a+\epsilon} g(\bar {\ob x}_1) \bar c_1$ if \eqref{equatifg} holds, where  $a$ and $\epsilon$ can be found in the definition of $f(t)$.
 By~\cite[Theorem~5.15]{GRSS}, $\kappa[\bar {\ob x}_1]$ can be considered as a ring of  polynomials in variable $\bar {\ob x}_1$. Thus  $\kappa[\bar {\ob x}_1]$ is a unique factorization domain. This means that
 $c$ is even and $\bar {\ob x}_1-\bar v_j, \bar {\ob x}_1+\bar v_j$ appear in $g(\bar {\ob x}_1)$, simultaneously. Note that $g(t)$ is a monic polynomial in variable $t$,   we have
$$g(\bar {\ob x}_1)=\bar {\ob x}_1^{a_1}\prod_{j=1}^{b_1}(\bar {\ob x}_1^2-\bar u_j)$$
for  some scalars $\bar u_i$'s in $\kappa$, where $b_1=c/2$ such that $a_1+2b_1=\ell$.\end{proof}

Thanks to  \eqref{funcf}, \eqref{reform}--\eqref{equatifg}, we can define
\begin{equation}\label{llc}  BC_{r, t}^{f}:=BC^{\rm aff}_{r, t}(\delta)/I,\end{equation} where
$I$ is the two-sided super
ideal generated by $f(\ob x_1)$ and $g(\bar {\ob x}_1)$. This is the same as the  level $\ell$ walled Brauer-Clifford superalgebra $BC_{\ell,r, t}$ defined in \cite[Definition~3.14]{GRSS}.

\begin{Theorem}\label{scycido} Suppose that $\kappa$ is an algebraically closed field with characteristic not two.
As $\kappa$-superalgebras,
$BC_{r, t}^{f} \cong  \End_{\OBC_\kappa^f(\delta)}(\downarrow^t\uparrow^r)$.
\end{Theorem}
\begin{proof} Since $\OBC_\kappa^f(\delta)$ is a quotient supercategory of $\AOBC_\kappa(\delta)$, we have the canonical epimorphism   from $\End_{\AOBC_\kappa(\delta) }(\downarrow^t\uparrow^r)$ to $\End_{\OBC_\kappa^f(\delta)}(\downarrow^t\uparrow^r)$, which is induced by the quotient superfunctor from   $\AOBC_\kappa(\delta)$ to  $\OBC_\kappa^f(\delta)$ \cite{CK}. Composing $\bar\varphi$ (see \eqref{isomofoabc}) with this epimorphism  yields an epimorphism
\begin{equation}\label{varphif}
\bar\varphi^f: {BC}_{r,t}^{\rm aff}(\delta)\longrightarrow  \End_{\OBC_\kappa^f(\delta)}(\downarrow^t\uparrow^r).
\end{equation}
We have  $f(\xdot)=0$ in $\OBC_\kappa^f(\delta)$. So  $ \bar\varphi^f(f(\ob x_1))=0$ in $\OBC_\kappa^f(\delta)$.
By \eqref{ggg}, we have $g(\xdotr)=\mathbf g=0$ in $\OBC_\kappa^f(\delta)$.
Hence we have  $\bar\varphi^f(g(\bar{\ob x}_1))=0$ in $\OBC_\kappa^f(\delta)$ by \eqref{bark}.
So, $\bar\varphi^f$ factors through $BC_{r, t}^{f}$.
 It results in   an induced surjective superalgebra homomorphism  $\tilde  \varphi^f: BC_{r, t}^{f} \rightarrow  \End_{\OBC_\kappa^f(\delta)}(\downarrow^t\uparrow^r)$. It is easy to check that $\tilde  \varphi^f$ sends a regular monomial of $BC_{r, t}^{f}$ (see \cite[Definition~3.15]{GRSS}) to a normally ordered dotted oriented Brauer-Clifford diagram, and moreover, the images of
two regular monomials are not equivalent.
By \cite[Corollary~3.16]{GRSS} and Theorem~\ref{Cyclotomic basis conjecture1}, we see that  $\tilde \varphi^f$ sends a basis of  $BC_{r, t}^{f}$ to a basis of $\End_{\OBC_\kappa^f(\delta)}(\downarrow^t\uparrow^r)$, forcing $\tilde \varphi^f$ to be  an isomorphism.
\end{proof}

As explained in \cite{CK}, the proof of Theorem~\ref{scycido} does not depend on the result of a basis of
 $BC_{r, t}^{f}$. Via Theorem~\ref{Cyclotomic basis conjecture1}, it can give a proof of the fact that the set of  all regular monomials of $BC_{r, t}^{f}$ is a basis of $BC_{r, t}^{f}$ when $\kappa$ is an algebraically closed field.  Finally,
one  can use Theorem~\ref{scycido} to give  a presentation of  $\End_{\OBC_\kappa^f(\delta)}(\downarrow^t\uparrow^r)$.

\end{document}